\documentclass[12pt,psfig]{article}
\usepackage{amsthm,amssymb,amsmath,amsfonts,amscd}
\usepackage{graphicx,epsfig,latexsym}
\usepackage{cite,calc,xcolor,subfigmat,mathrsfs,dsfont}
\usepackage{hyperref}
\usepackage{bclogo}
\usepackage{marginnote}
\usepackage{graphicx}
\usepackage{tikz}
\usepackage{bm}
\usepackage{enumitem}
\usepackage{calligra}

\newcommand{\Span}{{\rm span}}
\newcommand{\ad}{{\rm ad}}
\newcommand{\rank}{{\rm rank}}
\newcommand{\Null}{{\rm nullity}}
\newcommand{\0}{\mathbf{0}}
\newcommand{\x}{\mathbf{x}}
\newcommand{\y}{\mathbf{y}}
\newcommand{\z}{\mathbf{z}}

\newcommand{\ie}{{\em i.e.,} }
\newcommand{\eg}{{\em e.g.,} }

\DeclareMathOperator{\lcm}{lcm}
\newtheorem{thm}{Theorem}[section]
\newtheorem{cor}[thm]{Corollary}
\newtheorem{lem}[thm]{Lemma}
\newtheorem{prop}[thm]{Proposition}
\theoremstyle{definition}
\newtheorem{defn}[thm]{Definition}
\newtheorem{exm}[thm]{Example}
\newtheorem{rem}[thm]{Remark}

\numberwithin{equation}{section}
\def\Blem {\begin{lem}}
\def\Elem {\end{lem}}
\def\be {\begin{equation}}
\def\ee {\end{equation}}
\def\ba {\begin{eqnarray}}
\def\ea {\end{eqnarray}}
\def\bes {\begin{equation*}}
\def\ees {\end{equation*}}
\def\bas {\begin{eqnarray*}}
\def\eas {\end{eqnarray*}}
\def\bpr {\begin{proof}}
\def\epr {\end{proof}}
    \newcommand*{\Scale}[2][4]{\scalebox{#1}{$#2$}}%
    \newcommand{\Sum}[2]{\Scale[0.79]{\displaystyle\,\sum_{#1}^{#2}}}

\date{}
\begin{document}

\baselineskip=18pt
\renewcommand {\thefootnote}{ }

\begin{center}
\leftline{}
\vspace{-0.500 in}
{\Large \bf Further elements on hypernormal forms of non-resonant double Hopf singularities } \\ [0.3in]

{\large Majid Gazor\(^{a, b},\) Boumediene Hamzi\(^c\), and Ahmad Shoghi\(^{a, \dag}\)\footnote{$^\dag\,$Corresponding author. Phone: (98-31) 33913634; Fax: (98-31) 33912602;
Email: ahmad.shoghi@alumni.iut.ac.ir.  }}

\vspace{0.105in} {\small {\em \({}^{a}\)Department of Mathematical Sciences, Isfahan University of Technology,
\\[-0.5ex]
Isfahan 84156-83111, Iran.  Email: mgazor@iut.ac.ir and ahmad.shoghi@alumni.iut.ac.ir.\\[0.5ex] \(^b\)School of Mathematics, Institute for Research in Fundamental Sciences (IPM), \\[-0.5ex]
P.O. Box: 19395-5746, Tehran, Iran\\[0.5ex]
\({}^{c}\)Department of Computational and  Mathematical Sciences, Caltech, Pasadena/CA, USA. Email: boumediene.hamzi@gmail.com. }}

\vspace{0.05in}

\noindent
\end{center}

\vspace{-0.10in}

\baselineskip=16pt

\:\:\:\:\ \ \rule{5.88in}{0.012in}

\begin{abstract}
In this paper, we deal with hypernormal forms of non-resonant double Hopf singularities. We investigate the infinite level normal form classification of such singularities with nonzero radial cubic part. We provide a normal form decomposition of normal form vector fields in terms of planar-rotating and planar-radial vector fields. These facilitate the pattern recognition and analysis of the corresponding generalized homological maps. This paper is the first instance of  the normal form classification for generic non-resonant double Hopf singularities without structural symmetry.
\end{abstract}

\vspace{0.10in}


\vspace{0.10in} \noindent {\it 2010 Mathematics Subject Classification}:\, Primary: 34C20, 34A34, 16W5; Secondary: 68U99.

\noindent \rule{5.88in}{0.012in}

\vspace{0.2in}

\section{Introduction}

Normal form theory and center manifold theory provide us with two essential reduction techniques for the stability analysis and bifurcation control of nonlinear systems in the vicinity of a non-hyperbolic equilibrium. The center manifold is an invariant manifold of the differential (difference) system where the reduction of the differential on that gives rise to a reduction of the state space dimension. The computation of the center manifold leads to a quasilinear partial differential (nonlinear functional) equation which is not typically easy to solve. A low-degree approximation of the reduced system on the center manifold usually suffices for stability and bifurcation control analysis of the original system. Normal form theory facilitate an effective tool of simplifying the reduced systems on the center manifolds; \eg see \cite{AlgabaSIADS,GazorSadri,GazorSadriBog,GazorShoghiCIMP,GazorShoghiHarmonic,GazorShoghiMusic} and the references therein.

For  nonlinear systems with control bifurcations (see \cite{Krenercb}), a similar approach was used for the analysis and stabilization of systems with one or two uncontrollable modes in continuous and discrete-time \cite{Ka98,HBPhd,Krenercb,bh1,bh2,bh3,bh4,bh5,bh6,bh7,bh8,bh9,bh10,bh11,bh12,bh13,bh14,bh15,bh16,bh17,bh18,bh19,bh20,bh21,bh22,bh23,bh24,bh26}. This approach was generalized to systems with arbitrary number of uncontrollable modes by introducing the {\it Controlled Center Dynamics} in continuous time and discrete-time \cite{bh6,bh2}, as well as the {\it Controlled Center Systems} \cite{bh1,bh5}. The Controlled Center Dynamics and the Controlled Center Systems are reduced order control systems whose stabilizability properties determine the stabilizability properties of the full order system. The approach based on the Controlled Center Dynamics and the Controlled Center Systems can also be viewed as reduction techniques for some classes of controlled differential (difference) equations. A stabilizing controller can be designed based on the reduced order control system in normal form.

The classical way of finding normal forms of double Hopf singularities is to apply changes of coordinates to simplify the Taylor expansion of the vector field up to a given order via tuning the image of the homological map. Then, the remaining terms in the simplified system are \emph{resonant terms}. However, among the limitations of this approach when  applied to  problems in real life is that  the relations between the original parameters of the system and unfolding parameters are not easy to find. Finding simpler normal forms allows us to analyze dynamical systems more effectively. This is particularly important for the bifurcation analysis of degenerate cases and systems with structural symmetries. Note that the classical approach does not necessarily remove all possible terms from the Taylor series  expansion. In other words, the classical normal forms are not necessarily the simplest normal forms. The  simplest normal forms are unique representatives of the equivalent classes of differential systems under the group action of near-identity changes of state variables.  This gives rise to the normal form classification of the differential systems.

 The simplest normal forms of singular differential systems have an extensive literature of more than three decades \cite{Murd04,Ushiki,MurdBook,GaetaFurther} with alternative names such as the infinite level normal forms or unique normal forms; \eg see \cite{GazorShoghiCIMP,GazorMokhtariInt,GazorMokhtari,YuLeung03,YuBog,wlhj,GazorYuSpec,baidersanders}. Most of the results belong to two and three dimensional singularities; \eg see \cite{AlgabaHopfZ,YuTriple,YuHopfZero,YuBog,wlhj,Wang3DIJBC2017a,LiZhangWang14,KokubuWang,baidersanders,ChenNF,chendora,WangChenHopfZ,Algaba30,GaetaFurther,
GaetaPoincare,GaetaMathPhys}. For nonlinear control systems, a first step was taken in \cite{HamziCharNormalF} where inner-product normal forms were generalized to nonlinear control systems. There are very few published results on four dimensional singularities where they are not the simplest normal forms; \eg see \cite{2HopfNFLeung,Knobloch86,Yu2Hopf}. This paper provides the first results on the simplest normal forms of generic non-resonant double Hopf singularity without any structural symmetry. As for the bifurcation control is concerned, efficient changes of time rescaling and also the consideration of parametric differential systems are necessary; \eg see
\cite{GazorMoazeni,GazorMokhtari,GazorMokhtariInt,GazorSadri,GazorSadriBog,GazorYuSpec,GazorYuFormal,YuLeung03}.

\pagestyle{myheadings} \markright{{\footnotesize {\it  \hspace{2.3in} {\it Non-resonant double Hopf  singularity }}}}

{Let \(\mathcal{G}\) be a Lie subalgebra of four dimensional polynomial vector fields. Here, we define the Lie bracket by \([v, w]= vw-wv,\) where the vector fields are also considered differential operators. We intend to normalize a given vector field \(V\in \mathcal{G}\) of non-resonant double Hopf singularity. In this paper, a vector field \(V\) is associated with the ODE \(\frac{d\mathbf{y}}{dt}= V(\mathbf{y})\) and use vector fields and their associated differential systems interchangeably. Lie algebra structure here is relevant through the space of transformation generators. Each nonlinear vector field, say \(X(\y)\) with \(X(\0)=\0\) and no linear term in its Taylor expansion, is considered a transformation generator where it generates a nonlinear near-identity change of state variables. Near-identity here means that its linearization at the origin is the identity map. In fact, \(X(\y)\) generates the time-one map flow  \(\y=\y(1, \x)\) of the initial value problem
\be\label{TransGen}\dot{\y}(t, \x)=X(\y(t, \x)),\quad \y(0, \x)=\x;
\ee \eg see \cite{GazorYuSpec,KokubuWang,Wang3DIJBC2017a,MurdBook}. Therefore, the transformation \(\y=\y(1, \x)\) sends the vector field \(V(\y)\) into
\begin{eqnarray}\label{Exp}
 &\exp \ad_X V=\sum_{j=0}^{\infty}\frac{1}{j!}{\ad_X}^jV= V+ \ad_XV+ \frac{1}{2} \ad_X(\ad_XV)+ \cdots&
\end{eqnarray} in terms of the new coordinates \(\x\) where \(\ad_X V:= [X, V].\) This gives rise to a new  updated differential system as
\begin{eqnarray*}
 &\frac{d\mathbf{x}}{dt}= \exp \ad_X V|_{\mathbf{y}=\mathbf{x}}=\sum_{j=0}^{\infty}\frac{1}{j!}{\ad_X}^jV= V+ \ad_XV+ \frac{1}{2} \ad_X(\ad_XV)+ \cdots|_{\mathbf{y}=\mathbf{x}}.&
\end{eqnarray*} The idea is to apply a sequence of transformation generators like \(X\) to simplify \(V\) as much as possible. In practical computations, we usually deal with polynomial vector fields (as a finite Taylor expansion of a smooth system), and the vector field  \(\exp \ad_X V\) is truncated up to a certain desired grade.

Let \(\delta\) be a grading function, \ie the space \(\mathcal{G}\) is a \(\mathbb{Z}_{\geq 0}\)-graded Lie algebra. This means that \(\mathcal{G}= \sum_{k=0}^{\infty} \mathcal{G}_k\) (direct sum of \(k\)-grade homogeneous spaces) and \([\mathcal{G}_m, \mathcal{G}_k]\subseteq \mathcal{G}_{m+k}\) for all \(m, k\). Hence, we can write \(V=\sum_{k=0}^{\infty} V^{}_k\) where \(V_k\in \mathcal{G}_k\) for every \(k\in \mathbb{N}\cup\{0\}.\) Now we inductively define the generalized homological operators. Let
\bas &d^{n, 1}: \mathcal{G}_n \rightarrow \mathcal{G}_n\;\hbox{ be defined by } \; d^{n, 1}(X_{n}):= [ X_{n}, V^{}_0] \; \hbox{ for any grade } n.& \eas  These are the classical homological maps; \eg see \cite{MurdBook}. However, we call \(d^{n, 1}\) by a first level (generalized) homological map. The domain of this map corresponds with transformation generators while their image is called the first level removable spaces \(\mathcal{R}^{n,1}\). In other words, the transformation generators from the domain of the generalized homological maps can be used to simplify (projected) parts of \(V\) that belong to the removable spaces \(\mathcal{R}^{n,1}\). Therefore, the only remaining terms in the normal form belong to complement spaces \(\mathcal{C}^{n,1}\) that are complement to the removable spaces.

The idea of hyper-normalization is to use transformation generators in \(\ker d^{n-1, 1}:= \{X_{n-1}\in \mathcal{G}_{n-1}| [X_{n-1}, V_0]=0\}\) to enlarge the removable space \({\rm im}\, d^{n, 1}.\) Given this idea, we inductively define \(s\)-th level generalized homological map \(d^{n,s}: \ker(d^{n-1,s-1})\times \mathcal{G}_n \rightarrow \mathcal{G}_n\) for any \(2\leq s\leq n\) by
\bas
&d^{n,s}(X_{n-s+1}, \ldots, X_{n-1}, X_{n}):=[X_n, V_0]+[X_{n-1}, V_1]+\cdots+ [X_{n-s+1}, V_{s-1}],&
\eas where
\((X_{n-s+1}, X_{n-s+2}, \cdots, X_{n-1})\in \ker(d^{n-1,s-1})\); also see \cite[page 1015]{GazorYuSpec} and \cite{MurdBook,Murd04,SandersSpect}. Denote the image of the \(s\)-th level generalized homological map by \(\mathcal{R}^{n,s}\), \ie \(\mathcal{R}^{n,s}:= {\rm im}\, d^{n, s}\) is a \(s\)-level removable space. Transformation generators from \(\ker d^{n, 1}\) do not change terms of grade less than \(n\) as their Lie product with \(V_0\) is zero while they enlarge the removable spaces, \({\rm im} \,d^{n, 1}\subseteq {\rm im} \,d^{n, 2}.\) This is also true for \(\ker(d^{n-1,s-1})\) and generally, we have \({\rm im} \,d^{n, r}\subseteq {\rm im} \,d^{n, s}\)  for \(r\leq s.\) A complement image of \(d^{n,s}\) is denoted by \(\mathcal{C}^{n, s},\) \ie \(\mathcal{R}^{n,s}\oplus \mathcal{C}^{n,s}= \mathcal{G}_n.\) The spaces \(\mathcal{C}^{n, s}\) and \(\mathcal{R}^{n,s}\) are called the \(s\)-level complement and removable spaces, accordingly. The symbol \(\oplus\) stands for the direct sum of the spaces, \ie \(\mathcal{R}^{n,s}\cap \mathcal{C}^{n,s}=\{\0\}\) and for every \(V_n\in\mathcal{G}_n,\) there are
\bes w_n\in \mathcal{R}^{n,s}\;\hbox{ and } \; u_n\in \mathcal{C}^{n,s}\; \hbox{ so that } \; V_n=w_n+u_n.\ees


\begin{defn}
We call a vector field \(V=\sum_{k=0}^{\infty} V_k\) as a \(s\)-level normal form when all grade-homogenous vector fields \(v_k\) belong to the corresponding \(s\)-level complement spaces, \ie \(V_k\in \mathcal{C}^{k, s}\) for all \(k.\) When \(V_k\in \mathcal{C}^{k, k}\) for all \(k\geq 1,\) the vector field \(V=\sum_{k=0}^{\infty} V_k\) is called an infinite level normal form.
\end{defn}
\(N\)-Jet of a vector field \(v\) is defined as Taylor expansion of  \(v\) at the origin up to degree \(N\) and we denote it by \(J^N v\).
By \cite[Lemma 4.3]{GazorYuSpec}, there exist permissible changes of state variables to send the vector field \(V\) into its \(N\)-jet \(s\)-level normal form \(V^{(s)}=\sum_{n=0}^{N} V^{(s)}_n\) where \(V^{(s)}_n\in \mathcal{C}^{n, s}\).
Here, \(V^{(s)}_i= V^{(i)}_i\) for all \(i< s.\) We remark that the computation of \(s\)-level normal forms is usually facilitated by updating the generalized homological maps in each level of normalization, \ie \(d^{n,s}(X_{n-s+1}, \ldots, X_{n-1}, X_{n}):= \sum^{s-1}_{i=0}[ X_{n-i}, V^{(s-1)}_i]\); \eg see \cite{wlhj,SandersSpect,GazorYuSpec}.}

\paragraph{Statement of the problem} In this paper, we are concerned about the simplest (infinite level) normal form computation of the differential system
\begin{align}\label{pre-classical}
\textstyle\frac{d\mathbf{z}}{dt}=A\mathbf{z}+\mathcal{H}(\mathbf{z}), \quad \mathcal{H}(\0)=\0,
\end{align} so that \(\mathbf{z}:=(z_1, \overline{z}_1, z_2, \overline{z}_2)^\intercal\) and the diagonal matrix \(A\) is defined as
\begin{eqnarray*}
\textstyle A:=\mathrm{diag}\left(I\omega_1, -I\omega_1, I\omega_2, -I\omega_2\right),
\end{eqnarray*}
where \(\omega_1\) and \(\omega_2\) are non-resonant angular frequencies, \ie $\frac{\omega_1}{\omega_2}\notin\mathbb{Q}.$ For our convenience, we interchangeably use the notations
\bes (f_1, f_2, f_3, f_4)^\intercal\hbox{ and } f_1\frac{\partial}{\partial z_1}+f_2\frac{\partial}{\partial \overline{z}_1}+f_3\frac{\partial}{\partial z_2}+f_4\frac{\partial}{\partial \overline{z}_2}.\ees This gives rise a one-to-one correspondance between differential systems and vector fields, where we use vector fields and differential systems interchangeably. The vector field \(\mathcal{H}(\mathbf{z})\) represents nonlinear vector field in \(\mathbf{z}\) whose components are polynomials in \(\mathbb{C}[\mathbf{z}]\). Here, \(\mathbb{C}[\mathbf{z}]\) denotes the ring of polynomials over \(\mathbb{C}.\) Therefore, differential system \eqref{pre-classical} is a {\it non-resonant double Hopf singularity}. Despite literature of more than several decades old on the normal form classification of singularities, to the best of our knowledge, there have been no results on non-resonant double Hopf singularity with no structural symmetry beyond Poincare normal forms; \eg see \cite{Yu2Hopf,2HopfNFLeung,Knobloch86} for the classical results. There are near identity polynomial changes of state variables (see Theorem \ref{Thm1}) to transform \eqref{pre-classical} into a \(N\)-degree first level normal form given by
\begin{eqnarray}\label{FirstLevel}
\frac{d\mathbf{z}}{dt}=\sum_{m+n=0}^N  |z_1|^{2m} |z_2|^{2n} A_{m, n}\mathbf{z},
\end{eqnarray} where
\be\label{aijbij} A_{m, n}:=\mathrm{diag}\left( a^1_{m, n}+Ib^1_{m, n}, a^1_{m, n}-Ib^1_{m, n}, a^2_{m, n}+Ib^2_{m, n}, a^2_{m, n}-Ib^2_{m, n} \right)\hbox{ and } A_{0, 0}=A.\ee Here, the number \(N\) stands for an arbitrary large natural number. This is to skip formulations from formal power series into the polynomials of arbitrary large degrees. Define
\begin{equation}\label{E&Theta}
 \textstyle \mathscr{P}^k_{0,0}:= z_k\frac{\partial}{\partial z_k}+w_k\frac{\partial}{\partial w_k}, \quad \mathscr{R}_{0, 0}^k:= Iz_k\frac{\partial}{\partial z_k}-Iw_k\frac{\partial}{\partial w_k},\quad 
\end{equation} and \(w_k:= \overline{z}_k\) for  \(k= 1, 2.\) The vector fields \(\mathscr{P}^k_{0,0}\) and \(\mathscr{R}_{0, 0}^k\) are respectively called planar-radial (planar-Eulerian) and rotating vector fields.
These generate a Lie algebra defined by
\begin{equation}\label{L}
\mathcal{G}:=\Span \left\{{\rho_1}^{2m}{\rho_2}^{2n}\mathscr{P}^k_{0, 0}, {\rho_1}^{2m}{\rho_2}^{2n}\mathscr{R}_{0, 0}^k\,|\, m,n\in \mathbb{N}\cup\{0\}, \, k\in \{1, 2\}\right\}
\end{equation} for \(\rho_k:=|z_k|,\) \(k=1, 2\). Hence, vector fields associated with normal  form systems \eqref{FirstLevel} for every non-resonant double Hopf singularity fall into the Lie algebra \(\mathcal{G}.\) In fact, the vector field in equation \eqref{FirstLevel} is given as
\ba\label{Eq1.7}
&\frac{d\mathbf{z}}{dt}=\sum_{k=1}^{2}\sum_{m+n=0}^N  {\rho_1}^{2m}{\rho_2}^{2n}(a^k_{m, n}\mathscr{P}^k_{0, 0}+ b^k_{m,n}\mathscr{R}_{0, 0}^k), \quad a^{1}_{0,0}=a^{2}_{0,0}=0, b^1_{0,0}= \omega_1, b^2_{0,0}= \omega_2.&
\ea The use of planar-radial and planar-rotating vector fields substantially facilitates the pattern recognition of the generalized homological maps.

\begin{exm}[Further normalization of a Poincare normal form] In this example, we simplify a Poincare normal form system; also see \cite{GaetaPoincare,GaetaMathPhys,GaetaFurther}. Since only resonant terms of odd degrees may remain in Poincare normal forms, we assume that a Poincare normal form of a differential system in complex coordinates up to degree 6 is given by
  \begin{eqnarray}  \label{Exm0}
&\frac{d z_1}{dt}=Iz_1+z_1|z_1|^2+z_1|z_1|^4-2z_1|z_1|^2|z_2|^2,  \quad \frac{d z_2}{dt}=I\sqrt{3}z_2+z_2|z_1|^4, \;\frac{d w_2}{dt}= \frac{d \overline{z_2}}{dt},\; \frac{d w_1}{dt}= \frac{d \overline{z_1}}{dt},&
\end{eqnarray}
where \(w_k:=\overline{z}_k\) for \(k=1, 2,\) \(\omega_1:=1,\) and \(\omega_2:=\sqrt{3}.\) Since levels of normal forms correspond with the grading structure, we define grade of each homogeneous polynomial vector field as half of the standard degree minus one. Let \(V_0:=\left(Iz_1, -Iw_1, I\sqrt{3}z_2, -I\sqrt{3}w_2\right)^\intercal,\) \(V_1:= \left(z_1|z_1|^2, w_1|z_1|^2, 0, 0\right)^\intercal,\) \(V_2:=\left(z_1|z_1|^4-2z_1|z_1|^2|z_2|^2, w_1|z_1|^4-2w_1|z_1|^2|z_2|^2, z_2|z_1|^4,  w_2|z_1|^4\right)^\intercal.\)  Hence, \(V=V_0+V_1+V_2.\) Our goal is to simplify 5-degree polynomials \(V_2\) as much as possible through cubic part \(V_1\) of the normal form system. Let
\begin{equation*}
 X(\mathbf{z}):=-\left({z_1|z_2|^2}, {w_1|z_2|^2}, \frac{z_2|z_1|^2}{2}, \frac{w_2|z_1|^2}{2}\right)^\intercal.
\end{equation*} Then, we have
\be\label{Linear} \ad_{X}\, V_0=0, \quad\hbox{ \ie} X\in \ker d^{3,1}
\ee and
\begin{eqnarray}\label{Cubic}
&\ad_{X}\,V_1=-\left(-2z_1|z_1|^2|z_2|^2, -2w_1|z_1|^2|z_2|^2, z_2|z_1|^4, w_2|z_1|^4\right)^\intercal;&
\end{eqnarray}
see Section \ref{sec2} for how to choose \(X(\mathbf{z}).\) Equation \eqref{Linear} guarantees that transformation generator \(X\) does not create terms that had already been simplified using the linear part \(V_0\) of the system in Poincare normalization step. This, together with equation \eqref{Cubic}, certifies the simplification of terms appearing on the right hand side of equation \eqref{Cubic}. Let \(\mathbf{z}(t, \mathbf{u})\) be solution of the initial value problem \(\frac{d \mathbf{z}(t, \mathbf{u})}{dt}=X(\mathbf{z}(t, \mathbf{u})),\) \(\mathbf{z}(0, \mathbf{u})=\mathbf{u}.\) The near identity changes of state variables \(\mathbf{z}=\mathbf{z}(1, \mathbf{u})\) (time-one map of the flow \(\mathbf{z}(t, \mathbf{u})\)) sends the differential equation \eqref{Exm0} up to degree 6 to the following equation
\begin{eqnarray}&\nonumber\frac{d \mathbf{u}}{dt}=\exp\ad_{X(\mathbf{u})}\left(V_0(\mathbf{u})+V_1(\mathbf{u})+V_2(\mathbf{u})\right)=V_0(\mathbf{u})+V_1(\mathbf{u})+V_2(\mathbf{u})+\ad_{X(\mathbf{u})}V_1(\mathbf{u})+ {\sc h.o.t.}&\\\label{Exm0_1}
&=V_0(\mathbf{u})+V_1(\mathbf{u})+\left(u_1|u_1|^4, v_1|u_1|^4, 0,  0\right)^\intercal+ {\sc h.o.t.}.&
\end{eqnarray}
Here, \(\mathbf{u}:=(u_1, v_1, u_2, v_2)\) stands for the new state variables. We ignore higher order terms. Then, Equation \eqref{Exm0_1} illustrates the second level form of Equation \eqref{Exm0}, \ie \(s=2\). Since there is no four degree terms in the system, there is no removable terms within its five-degree terms via any changes of state variables. Hence, Equation \eqref{Exm0_1} is the simplest five degree polynomial normal form.
\end{exm}

Throughout this paper, we assume that there is a non-zero cubic assumption in the first level normal forms \eqref{Eq1.7} given by
\begin{equation}\label{GenCon}
\left|a^1_{0, 1}\right|+\left|a^2_{0, 1}\right|+\left|a^1_{1, 0}\right|+\left|a^2_{1, 0}\right|\neq 0.
\end{equation} 
These four numbers represent cubic coefficients of the Poincare normal form. If all these coefficients were to be zero, the quartic normal form system only included rotational vector fields with zero radial components. It is often assumed that these highly degenerate cases rarely occur in real life applications, unless there is a structural symmetry within the system that enforces these specific degenerate cases. Treatment of these degenerate cases is beyond the scope of this paper.

\section{Normal forms and Lie structure constants }\label{sec2}

This section is devoted to the first level normal forms and an algorithmic approach for deriving the cubic normal form coefficients; also see \cite{Knobloch86,2HopfNFLeung,Yu2Hopf}.

\begin{thm}\label{Thm1}
For any differential system \eqref{pre-classical}, there are near-identity changes of state variables such that they transform system \eqref{pre-classical} into a \(2N+1\)-degree truncated first level normal form given in complex coordinates by
\ba\label{classicalNF}
\hspace{-1 cm}&\dot{z}_1=\sum^N_{m+n= 0}(a^1_{m,n}+Ib^1_{m,n}) z_1^{m+1}w_1^m z_2^n w_2^n,\quad \dot{z}_2=\sum^N_{m+n= 0}(a^2_{m,n}+Ib^2_{m,n}) z_1^{m}w_1^m z_2^{n+1} w_2^n,&
\ea
 where \({w}_k=\overline{{z}_k}\) , \(b^k_{m,n}, a^k_{m,n}\in \mathbb{R},  b^k_{0,0}= \omega_k\) and \(a^k_{0, 0}=0\) for \(k=1, 2.\)
\end{thm}
\begin{proof} The proof is well-known. For example see \cite{Knobloch86,Yu2Hopf,LiZhangWang14,MurdBook}.
\end{proof}

From now on, we treat the vector fields of type \eqref{classicalNF}. The family of vector fields given by \eqref{classicalNF} constitutes a Lie subalgebra for the family of all double Hopf singularities. Since the linear part of the vector field belongs to the centralizer of the Lie subalgebra, it will have no further influence on the normalization  beyond the first level step. Thus, the remaining results in this paper readily hold for the resonant and non-resonant double Hopf singularities given by  \eqref{classicalNF}.
For each \(m,n\in\mathbb{Z}_{\geq 0}\!:=\! \mathbb{N}\cup\{0\},\) define
\ba\label{def-E-and-Theta}
\mathscr{P}^{k}_{m,n}:= |z_1|^{2m}|z_2|^{2n}  \mathscr{P}_{0,0}^k,  \hbox{ and } \quad
\mathscr{R}_{m,n}^k:= |z_1|^{2m}|z_2|^{2n} \mathscr{R}_{0,0}^k, \hbox{ for } k=1, 2,
\ea and let
\begin{equation}\label{L}
\mathcal{G}:=\Span\left\{\Theta+\sum (a^{k}_{i, j} \mathscr{P}^{k}_{i, j}+b^{k}_{i, j} \mathscr{R}_{i, j}^k) \,\big|\, a^{k}_{i, j}, b^{k}_{i, j}\in \mathbb{R}, k=1, 2,\, i,j\in\mathbb{Z}_{\geq 0}, i+j\geq 1\right\}.
\end{equation}
\begin{prop}
The vector space \(\mathcal{G}\) defined in \eqref{L} over the real numbers \(\mathbb{R}\) constitutes an infinite dimensional real Lie algebra whose structure constants follow (for all \(k, l, m, n\in \mathbb{Z}_{\geq 0}\))
\begin{eqnarray}
&{\left[\mathscr{P}^{1}_{m,n} ,\mathscr{P}^{1}_{i,j}\right]\! =\!2(i\!-\!m)\mathscr{P}^1_{m+i,n+j},\qquad
\left[\mathscr{P}^{2}_{m,n} ,\mathscr{P}^{2}_{i, j}\right]\! =\!2(j\!-\!n)\mathscr{P}^{2}_{m+i,n+j},}\nonumber&\\
&\left[\mathscr{P}^{1}_{i, j}, \mathscr{P}^{2}_{m,n}\right]\! =\!2m\mathscr{P}^{2}_{m+i,n+j}\!-\!2j\mathscr{P}^{1}_{m+i,n+j},\qquad \left[\mathscr{P}^{1}_{i, j}, \mathscr{R}_{m,n}^{k}\right]\! =\!2m \mathscr{R}_{m+i,n+j}^{k},& \label{StructureConstants}\\\nonumber
&\left[\mathscr{P}^{2}_{i, j}, \mathscr{R}_{m,n}^{k}\right]\!=\!2n \mathscr{R}_{m+i,n+j}^k,\qquad \left[\mathscr{R}_{m,n}^{k},\mathscr{R}_{i, j}^{l}\right]\! =\!0,\quad \hbox{  for }\; k,l=1, 2.&
\end{eqnarray}
\end{prop}
\begin{proof}
Using the definition for the Lie bracket \([v, w]:= vw-wv,\) we have
\ba\nonumber
\left[f(\z)\mathscr{P}^k_{0, 0}, g(\z)\mathscr{P}^l_{0, 0} \right]&=&\langle \nabla g, \mathscr{P}^k_{0,0}\rangle f(\z) \,\mathscr{P}^l_{0,0}-\langle\nabla f, \mathscr{P}^l_{0,0}\rangle g(\z) \,\mathscr{P}^k_{0,0},
\\\nonumber
\left[f(\z)\mathscr{R}^k_{0, 0}, g(\z)\mathscr{R}^l_{0, 0} \right]&=&\langle\nabla g, \mathscr{R}^k_{0,0}\rangle f(\z) \,\mathscr{R}^l_{0,0}-\langle\nabla f, \mathscr{R}^l_{0,0}\rangle g(\z)\, \mathscr{R}^k_{0,0},
\\\label{StructGrad}
\left[g(\z)\mathscr{P}^k_{0,0}, f(\z)\mathscr{R}^l_{0, 0}\right]&=&\langle\nabla f, \mathscr{P}^k_{0,0}\rangle g(\z) \,\mathscr{R}^l_{0,0}-\langle\nabla g, \mathscr{R}^l_{0,0}\rangle f(\z) \,\mathscr{P}^k_{0,0},\ k,l =1, 2.
\ea Then, the proof simply follows notation \eqref{def-E-and-Theta}.
\end{proof}


In the second level normalization step, we intend to inductively simplify terms with grade \(n\) from \(v^{(1)}\). We first derive reduced representative squared matrices for the second level generalized homological maps, where the complement images of the generalized homological maps can be concluded from the representative squared matrices.
Recall that every  differential system \eqref{classicalNF} corresponds with a vector field \(v^{(1)}\) from the Lie algebra \(\mathcal{G}\) and is written as
\begin{eqnarray}\label{eq2}
&v^{(1)}:=\sum^2_{k=1}\sum^N_{m+n=0}b^k_{m,n}\mathscr{R}_{m,n}^k+\sum^2_{k=1}\sum^N_{m+n= 1}a^{k}_{m,n}\mathscr{P}^{k}_{m,n}, &
\end{eqnarray} for some coefficients \(a^{k}_{m,n}, b^k_{m,n}\in \mathbb{R}\) and sufficiently large natural number \(N.\) Now define the grading function $\delta$ by
\begin{align}\label{grade}
\delta\left(\mathscr{P}^k_{m,n} \right) =m+n, \ \delta\left(\mathscr{R}_{m,n}^k \right) =m+n,\ k=1,2.
\end{align}

\begin{lem} A reduced representation matrix for the second level generalized homological map \(d^{n, 2}\) with respect to the vector space bases in \eqref{def-E-and-Theta}
follows
\begin{equation}\label{An1}
\mathcal{A}_{n, 1}:=-\begin{bmatrix}
\mathcal{A}_{n,1}^{0, 0}& \0                       & \cdots   & \0 \\
\mathcal{A}_{n,1}^{0, 1}& \mathcal{A}_{n,1}^{1, 0} & \ddots   & \vdots\\
\0                      & \mathcal{A}_{n,1}^{1, 1} & \ddots   & \0 \\
\vdots                  & \ddots                       &\ddots& \mathcal{A}_{n,1}^{n-1, 0}\\
\0                      & \cdots                   & \0       & \mathcal{A}_{n,1}^{n-1, 1}\\
\end{bmatrix}.
\end{equation} Here, the block matrices \(\mathcal{A}_{n,1}^{j, 0}\) and  \(\mathcal{A}_{n,1}^{j, 1}\) for \(0\leq j\leq n-1\) are defined by
\begin{eqnarray}
&&\hspace{-0.5 cm}\label{BAn0}\hspace{-1 cm}\Scale[0.93]{\mathcal{A}_{n,1}^{j, 0} :=\! 2\!\begin{bmatrix}
j a_{0,1}^1\!+\!(n\!-\!j\!-\!1) a_{0,1}^2 & -a_{0,1}^1& 0 & 0 \\
0 &  j a_{0,1}^1\!+\! (n\!-\!j\!-\!2)a_{0,1}^2 &0 & 0 \\
0 & -b_{0,1}^1 &  j a_{0,1}^1\!+\! (n\!-\!j\!-\!1)a_{0,1}^2  & 0\\
0 & -b_{0,1}^2 & 0 & j a_{0,1}^1\!+\! (n\!-\!j\!-\!1)a_{0,1}^2
\end{bmatrix}}\!,\\
&&\hspace{-0.5 cm}\label{BAn1}\hspace{-1 cm}\Scale[0.9]{\mathcal{A}_{n,1}^{j, 1}\! :=\! 2\!\begin{bmatrix}
(j\!-\!1) a_{1,0}^1\!+\!(n\!-\!j\!-\!1)a_{1,0}^2      &0  & 0 & 0 \\
-a_{1,0}^2 &  j a_{1,0}^1\!+\! (n\!-\!j\!-\!1)a_{1,0}^2   &0  & 0 \\
-b_{1,0}^1 & 0 &  j a_{1,0}^1\!+\! (n\!-\!j\!-\!1)a_{1,0}^2  & 0\\
-b_{1,0}^2 & 0 & 0 &  j a_{1,0}^1\!+\! (n\!-\!j\!-\!1)a_{1,0}^2
\end{bmatrix}}.
\end{eqnarray}
\end{lem}
\bpr
The number of terms that are typically simplified in the second level normal forms are equal to \(\rank\,d^{n, 2}\), \ie
\({\rm im}\, d^{n, 1}= \{0\}.\) More precisely, we have
\begin{equation*}
\ker d^{n-1, 1}=\{x\in \mathcal{G}_{n-1}\,|\, [x, \Theta]=0\}=\Span\{\mathscr{P}^k_{j, n-1-j}, \mathscr{R}^k_{j, n-1-j}, k=1, 2, 0\leq j\leq n-1\}=\mathcal{G}_{n-1}.
\end{equation*}
Further, \(d^{n, 2}(X_{n-1}, X_{n})=[X_{n-1}, v_1^{(1)}]+[X_{n}, v_0^{(1)}]=[X_{n-1}, v_1^{(1)}]\) where \((X_{n-1}, X_{n})\in \mathcal{G}_{n-1}\times  \mathcal{G}_{n}.\)
Here we use the method of {\it indeterminate coefficients} to solve the generalized homological equations. Let
\begin{eqnarray}\label{Xn}
  &X_{n-1}=\sum_{j=0}^{n-1}(\alpha^1_j\mathscr{P}^1_{j, n-1-j}+\alpha^2_j\mathscr{P}^2_{j, n-1-j}+\beta^1_j\mathscr{R}^1_{j, n-1-j}+\beta^2_j\mathscr{R}^2_{j, n-1-j})&
\end{eqnarray}
be a \(n-1\)-grade transformation generator for normalizing the updated \(n\)-grade terms of \(v_{n}^{(1)}:=\sum_{j=0}^{n}(\hat{a}^{1}_j\mathscr{P}^1_{j, n-j}+\hat{a}^{2}_j\mathscr{P}^2_{j, n-j}+\hat{b}^{1}_j\mathscr{R}^1_{j, n-j}+\hat{b}^{2}_j\mathscr{R}^2_{j, n-j})\) in the second level normalization step. Here, the assumption is that lower grade terms \(v_{d}^{(1)}\)
(\(d<n\)) have already been simplified to \(v_{d}^{(2)}\)  for \(d<n\). In other words, the second level normal form has been partially performed upto grade \(n-1.\) For simplicity of notations throughout this paper, the notations for the updated coefficients \(\hat{a}^{1}_j,\) \(\hat{a}^{2}_j,\) \(\hat{b}^{1}_j,\) and \(\hat{b}^{2}_j\) are replaced with \(a^{1}_j,\) \(a^{2}_j,\) \(b^{1}_j,\) \(b^{2}_j\). Here, \(\alpha^1_j,\) \(\alpha^2_j,\) \(\beta^1_j\) and \(\beta^2_j\) for \(0\leq j\leq n-1\) are the real unknown constants to be determined. Thus, \(\left[X_{n-1}, \sum_{k=1}^{2} \sum_{j=0}^{1}\left(a^k_{j, 1-j} \mathscr{P}^k_{j, 1-j}+b^k_{j, 1-j}\mathscr{R}^k_{j, 1-j}\right)\right]\) is given by
\begin{eqnarray}
&\nonumber\hspace{-0.5 cm}\Sum{l, k=1}{2}\Big(\left[\alpha^l_0\mathscr{P}^l_{0, n-1}+\beta^l_0\mathscr{R}^l_{0, n-1},  a_{0, 1}^k\mathscr{P}^k_{0, 1}+b_{0, 1}^k\mathscr{R}^k_{0, 1}\right]+\left[\alpha^l_{n-1}\mathscr{P}^l_{n-1, 0} + \beta^l_{n-1}\mathscr{R}^l_{n-1, 0},  a_{1, 0}^k\mathscr{P}^k_{1, 0} + b_{1, 0}^k\mathscr{R}^k_{1, 0}\right]&\\
&\nonumber\hspace{-0.5 cm}+\left(\left[\alpha^l_0\mathscr{P}^l_{0, n-1} + \beta^l_0\mathscr{R}^l_{0, n-1},  a_{1, 0}^k\mathscr{P}^k_{1, 0} + b_{1, 0}^k\mathscr{R}^k_{1, 0}\right] + \left[\alpha^l_1\mathscr{P}^l_{1, n-2} + \beta^l_1\mathscr{R}^l_{1, n-2},  a_{0, 1}^k\mathscr{P}^k_{0, 1} + b_{0, 1}^k\mathscr{R}^k_{0, 1}\right]\right) &\\
&\nonumber\hspace{-0.5 cm}+\left( \left[\alpha^l_1\mathscr{P}^l_{1, n-2} + \beta^l_1\mathscr{R}^l_{1, n-2},  a_{1, 0}^k\mathscr{P}^k_{1, 0} + b_{1, 0}^k\mathscr{R}^k_{1, 0}\right] + \left[\alpha^l_2\mathscr{P}^l_{2, n-3} + \beta^l_2\mathscr{R}^l_{2, n-3},  a_{0, 1}^k\mathscr{P}^k_{0, 1} + b_{0, 1}^k\mathscr{R}^k_{0, 1}\right] \right) + \cdots&\\
&\hspace{-0.9 cm}\Scale[0.95]{+\left(\left[\alpha^l_{n-2}\mathscr{P}^l_{n-2, 1} \!+\! \beta^l_{n-2}\mathscr{R}^l_{n-2, 1},  a_{1, 0}^k\mathscr{P}^k_{1, 0}\! +\! b_{1, 0}^k\mathscr{R}^k_{1, 0}\right]\! +\! \left[\alpha^l_{n-1}\mathscr{P}^l_{n-1, 0}\!+\! \beta^l_{n-1}\mathscr{R}^l_{n-1, 0},  a_{0, 1}^k\mathscr{P}^k_{0, 1} \!+\! b_{0, 1}^k\mathscr{R}^k_{0, 1}\right]\right)}\Big).\label{Lie-Bracket}&
\end{eqnarray}
For every \(i, j, l\) and \(m\in \mathbb{Z}_{\geq 0}\), the structure constants \eqref{StructureConstants} leads to
\begin{eqnarray*}\label{LieBracket}
  &\frac{1}{2}\left[\sum_{k=1}^{2}a_{l, m}^k\mathscr{P}_{l, m}^k+b_{l, m}^k\mathscr{R}_{l, m}^k, \sum_{k=1}^{2} a_{i,j}^k\mathscr{P}_{i,j}^k+b_{i,j}^k\mathscr{R}_{i,j}^k\right]=&\\
  &\Scale[0.96]{\hspace{-0.20 cm}\sum_{k=1}^{2}\!\left(\left((ia_{l, m}^1\!+\!ja_{l, m}^2)a_{i, j}^k\!-\!(la_{i, j}^1\!+\!ma_{i, j}^2)a_{l, m}^k\right)\!\mathscr{P}^k_{i+l, j+m}\!+\!\left((ia_{l, m}^1\!+\!ja_{l, m}^2)b_{i, j}^k\!-\!(la_{i, j}^1\!+\!ma_{i, j}^2)b_{l, m}^k\right)\!\mathscr{R}^k_{i+l, j+m}\right)}.&
\end{eqnarray*}
Hence, the solution to the normalization of terms in \(v_{n}^{(1)}\) leads to linear system
\begin{equation}\label{SysEqu}
\Scale[0.97]{ \mathcal{A}_{n, 1}(\alpha_0^1, \alpha_0^2, \beta_0^1, \beta_0^2,\cdots, \alpha_{n-1}^1, \alpha_{n-1}^2, \beta_{n-1}^1, \beta_{n-1}^2)^\intercal\!+\!(a_{0, n}^1, a_{0, n}^2, b_{0, n}^1, b_{0, n}^2,\cdots, a_{n, 0}^1, a_{n, 0}^2, b_{n, 0}^1, b_{n, 0}^2)^\intercal=0},
\end{equation} for constants \(\alpha_j^k, \beta_j^k\) where \(k=1, 2, 0\leq j\leq n-1,\) and the block matrix \(\mathcal{A}_{n, 1}\) is defined by \eqref{An1}.
\epr

Since \(\mathcal{A}_{n,1}\) is a $ 4(n+1)\times 4n$ matrix, \(\rank\, \mathcal{A}_{n,1}\leq 4n.\) Recall that \(n\)-grade homogeneous terms in \(\mathcal{G}_n\) belong to the \(4(n+1)\)-dimensional linear space spanned by \({\left\{\mathscr{P}^1_{j, n-j}, \mathscr{P}^2_{j, n-j}, \mathscr{R}^1_{j, n-j}, \mathscr{R}^2_{j, n-j}| 0\leq j\leq n\right\}}.\)
Hence, for each \(n\geq 2,\) \(n\)-grade terms of \(v^{(1)}\) are not necessarily removable in the second level hyper-normalization step. From equations \eqref{An1}, \eqref{BAn0} and \eqref{BAn1}, rank of matrix \(\mathcal{A}_{n, 1}\) depends to the values for \(a_{1, 0}^1,\) \(a_{1, 0}^2,\) \(a_{0, 1}^1\) and \(a_{0, 1}^2.\) These enforce the consideration of three different cases:
\begin{enumerate}[label=\Roman*]
\item\!. \label{I} \(\frac{a_{0, 1}^1}{a_{0, 1}^2}, \frac{a_{1, 0}^1}{a_{1, 0}^2}\in \mathbb{Q},\) \(a_{0, 1}^1a_{0, 1}^2<0\) and \(a_{1, 0}^1a_{1, 0}^2<0\);
\item\!. \label{II} either \(\frac{a_{1, 0}^1}{a_{1, 0}^2}\in \mathbb{Q}^c,\) or \(\left(a_{0,1}^1a_{0, 1}^2>0, a_{1, 0}^1a_{1, 0}^2>0, \frac{a_{0, 1}^1}{a_{0, 1}^2}\in
  \mathbb{Q}^c\right)\);
\item\!. \label{III} \(a^1_{0, 1}a^2_{0, 1}a_{1, 0}^1a_{1, 0}^2=0.\)
\end{enumerate}
Here we select a complement (choice of style) for the image of the generalized homological operator \(d^{n, 2}\) based on the simplicity of computation in equation \eqref{SysEqu}.  In this regard, we may reduce the matrix \(\mathcal{A}_{n, 1}\) to an invertible \((\rank\, \mathcal{A}_{n,1}\times \rank\, \mathcal{A}_{n,1})\)-matrix by removing some of its linearly dependent rows and columns, where it is then denoted by \(\hat{\mathcal{A}}_{n, 1}\). We remark that we will not need to explicitly calculate inverse matrix for \(\hat{\mathcal{A}}_{n, 1}.\) All variables of the reduced equations of \eqref{SysEqu} are recursively computed. These will be discussed  for all three cases I, II and III.

\subsection{Case I: When \(\frac{a_{0, 1}^1}{a_{0, 1}^2}, \frac{a_{1, 0}^1}{a_{1, 0}^2}\in \mathbb{Q},\) \(a_{0, 1}^1a_{0, 1}^2<0\) and \(a_{1, 0}^1a_{1, 0}^2<0\).}\label{subseccaseI}
\begin{lem}\label{lem1}
Let \(r+s=p+q,\)
\begin{equation}\label{Conditions}
\textstyle\frac{a_{0,1}^1}{a_{0,1}^2}=-\frac{q}{p}, \frac{a_{1, 0}^1}{a_{1, 0}^2}=-\frac{s}{r}\in \mathbb{Q}, \text{\quad where\quad} p,q, r, s\in \mathbb{N}\, \text{and}\,\, \gcd(p, q)= \gcd(r, s)=1.
\end{equation}
Then, there are near-identity transformations so that the first level normal form \(v^{(1)}\) can be transformed into the \(2N+1\)-degree truncated second level normal form vector field given by
\begin{eqnarray}
&\hspace{-0.4 cm}\notag\Scale[0.92]{\Theta\!+\!\Sum{k\in\mathbb{N}_2}{}\Sum{i+j=1}{}(a^k_{i, j}\mathscr{P}^k_{i, j}\!+\!a^k_{i j}\mathscr{R}^k_{i, j})\!+\!\hspace{-0 cm}\Sum{m\in\mathbb{N}_2}{}\Sum{j\in H_1}{}\big(\delta_1 \mathscr{P}^1_{j, m(p+q)\!+\!1\!-\!j}\!+\!a^2_{j, m(p+q)\!+\!1\!-\!j}\mathscr{P}^2_{j, m(p+q)+1-j}\!+\!\Sum{k\in\mathbb{N}_2}{} b^k_{j, m(p+q)\!+\!1\!-\!j }\mathscr{R}^k_{j, m(p+q)\!+\!1\!-\!j}\big)}&\\ [-4pt]
&\nonumber\Scale[0.92]{\!+\!\Sum{k\in\mathbb{N}_2}{}(a^k_{2, 0} \mathscr{P}^k_{2, 0}\!+\!b^k_{2, 0} \mathscr{R}^k_{2, 0})\!+\!\Sum{m=2}{\infty}\,\Sum{j\in H_2}{} (\delta_2 \mathscr{P}^2_{mp, mq+2}\!+\!\Sum{k\in\mathbb{N}_2}{} (a^k_{j, m(p+q)+2-j}\mathscr{P}^k_{j, m(p+q)+2-j}\!+\! b^k_{j, m(p+q)+2-j}\mathscr{R}^k_{j, m(p+q)+2-j}))}&\\[-4pt]
&\hspace{-1 cm}\Scale[0.92]{ +a^2_{0, 2} \mathscr{P}^2_{0, 2}\!+\!\!\Sum{j\in H_3}{} (\delta_3 \mathscr{P}^2_{p, q+2}\!+\!\!\Sum{k\in\mathbb{N}_2}{} (a^k_{j, p+q+2-j}\mathscr{P}^k_{j, p+q+2-j}\!+\! b^k_{j, p+q+2-j}\mathscr{R}^k_{j, p+q+2-j}))\!+\!\!\!\Sum{m\in\mathbb{Z}_{\geq 0}}{}\Sum{k\in\mathbb{N}_2}{}\Sum{j\in H_4}{} \big(a^k_{j,0}\mathscr{P}^k_{j, 0}\!+\!b^k_{j, 0}\mathscr{R}^k_{j, 0}\big)}&\label{SecondLevel1}
\end{eqnarray}
where \(H_4:=\{n|m(p+q)+3\leq n\leq (m+1)(p+q)\}.\) Since we are dealing with truncated normal forms, all infinite sums have finite nonzero terms upto degree \(2N+1\). Further,
\bes H_1:=\{mp, mr+1\},\quad H_2:=\{m(p+q)+2\}, \quad (\delta_1, \delta_2):=(0, a^2_{mp,mq+2}) \quad\hbox{ for  } p\leq r,\ees
and \(H_1:=H_2:=\{mp\}, (\delta_1, \delta_2):=(a^1_{j, m(p+q)+1-j}, 0)\) when \(p>r\). For \(p\leq r+1,\) \(\delta_3:=a^2_{mp,mq+2}, H_3:=\{m(p+q)+2\}\) while
if \(p>r+1,\) we have \(\delta_3:=0\) and \(H_3:=\{mp\}.\)
\end{lem}
\begin{proof}
Let \(n=2.\) The second and ninth rows of \(\mathcal{A}_{2,1}\) are zero vectors. Thus, terms \(\mathscr{P}^2_{0, 2}\) and \(\mathscr{P}^1_{2, 0}\) are not necessarily simplified. Hence, these zero rows of matrix \(\mathcal{A}_{2,1}\) can be simply omitted. The first seven rows of the remaining matrix are linearly independent. Then, the last three rows can be neglected to obtain the invertible matrix \(\hat{\mathcal{A}}_{2, 1}\); indeed, \(\rank\,\mathcal{A}_{2,1}= \rank\,\hat{\mathcal{A}}_{2, 1}=7.\) The latter rows correspond with terms \( a^2_{2,0}\mathscr{P}^2_{2,0}, b^1_{2,0}\mathscr{R}^1_{2,0}\) and \( b^2_{2,0}\mathscr{R}^2_{2,0}.\) Thereby, the simplification of these terms are not attempted in the second level normalization step. Next, let \( \alpha^1_{0,1}:=0 \) and
\begin{eqnarray}
&\alpha^2_{0,1}= -\dfrac{a^1_{0, 2}}{2a^1_{0, 1}},\, \beta^1_{0, 1}= \dfrac{a^1_{0, 1}b^1_{0, 2}-a^1_{0, 2}b^1_{0, 1}}{2a^1_{0, 1}a^2_{0, 1}},\,
 \beta^2_{0, 1}= \dfrac{a^1_{0, 1}b^2_{0, 2}-a^1_{0, 2}b^2_{0, 1}}{2a^1_{0, 1}a^2_{0, 1}},\label{n=2Coefs1}&\\
&(\alpha^1_{1, 0}, \alpha^2_{1, 0}, \beta^1_{1, 0}, \beta^2_{1, 0})^\intercal = (\mathcal{A}_{2, 1}^{1, 0})^{-1}\left((a_{1, 1}^1, a_{1, 1}^2, b_{1, 1}^1, b_{1, 1}^2)^\intercal -\mathcal{A}_{2, 1}^{0, 1}(\alpha^1_{0,1}, \alpha^2_{0,1}, \beta^1_{0, 1},\beta^2_{0, 1})^\intercal \right).\label{n=2Coefs2}&
\end{eqnarray} Remark that the matrix \(\mathcal{A}_{2, 1}^{1, 0}\) here is invertible due to \(a_{0, 1}^1a_{0, 1}^2<0.\) By the hypothesis, we have
\begin{equation*}
a_{0, 1}^1p+a_{0, 1}^2q=a_{1, 0}^1r+a_{1, 0}^2s=0.
\end{equation*} Let \(n=m(p+q)+1=m(r+s)+1\) for a \(m\in \mathbb{N}\) and \(n>2.\) Hence, we have
\begin{equation}\label{eq3} mra_{1, 0}^1+(n-1-mr)a_{1, 0}^2=mpa_{0, 1}^1+(n-1-mp)a_{0, 1}^2=0.
\end{equation}
By the substitution of these values into matrix \(\mathcal{A}_{n, 1},\) the first, third and fourth columns of \(\mathcal{A}_{n, 1}^{mp, 0}\) and the second, third and fourth columns of \(\mathcal{A}^{mr, 1}_{n, 1}\) are zero vectors. Here, we may consider two cases:
\bes
\text{(i)}\,\, p\leq r\qquad \hbox{ and } \qquad \text{(ii)}\,\, p>r.
\ees
Case (i): \(\rank\, \mathcal{A}_{n,1}= 4n-2\) can be obtained from the column reduced echelon form for \(\mathcal{A}_{n, 1}.\) Therefore, \(4n-2\) number of terms of grade \(n\) from the first level normal form can be simplified. Here, terms \(\mathscr{P}^1_{mp, n-mp}, \mathscr{P}^1_{mr+1, n-mr-1}, \mathscr{P}^k_{j, n-j}\) and \(\mathscr{R}^k_{j, n-j}\) for \(k=1, 2, 0\leq j\leq n, j\neq mp, mr+1\) can be normalized while \(\mathscr{P}^2_{j, n-j}, \mathscr{R}^1_{j, n-j}, \mathscr{R}^2_{j, n-j}\) for \(j=mp, mr+1\) may remain in the second level normal form. The matrices \(\mathcal{A}_{n,1}^{i, 0}\) and \(\mathcal{A}_{n,1}^{j, 1}\) are invertible for \(i\neq mp\) and \(j\neq mr.\) Then for \(0\leq i\leq  mp-1,\) take an ascending recursive relation
\begin{eqnarray}\label{AscRec}
(\alpha_i^1, \alpha_i^2, \beta_i^1, \beta_i^2)^\intercal=(\mathcal{A}_{n, 1}^{i, 0})^{-1}\left((a^1_{i, n-i}, a^2_{i, n-i}, b^1_{i, n-i}, b^2_{i, n-i})^\intercal-\mathcal{A}_{n, 1}^{i-1, 1}(\alpha_{i-1}^1, \alpha_{i-1}^2, \beta_{i-1}^1, \beta_{i-1}^2)^\intercal\right),
\end{eqnarray} where \((\alpha_{-1}^1, \alpha_{-1}^2, \beta_{-1}^1, \beta_{-1}^2)=\0.\) For \(mr+2\leq i\leq n,\) the descending recursive relation
\begin{eqnarray}\label{DecRec}
(\alpha_{i-1}^1, \alpha_{i-1}^2, \beta_{i-1}^1, \beta_{i-1}^2)^\intercal=(\mathcal{A}_{n, 1}^{i-1, 1})^{-1}\left((a^1_{i, n-i}, a^2_{i, n-i}, b^1_{i, n-i}, b^2_{i, n-i})^\intercal-\mathcal{A}_{n, 0}^{i, 0}(\alpha_{i}^1, \alpha_{i}^2, \beta_{i}^1, \beta_{i}^2)^\intercal\right)
\end{eqnarray} is implemented and \((\alpha_n^1, \alpha_n^2, \beta_n^1, \beta_n^2)=0.\)
Now we take
\begin{eqnarray}\label{219}
&\alpha_{mp}^2:=-\frac{1}{2a_{0, 1}^1}\left(a_{mp, mq+1}^1-2\alpha_{mp-1}^1(a_{1, 0}^1(mp-2)+a_{1, 0}^2(mq+1))\right),&\\\label{220}
&\alpha_{mr}^1:=-\frac{1}{2a_{1, 0}^1}\left(a_{mr+1, ms}^1-2\alpha^1_{mr+1}((mr+1)a_{0, 1}^1+(ms-1)a^2_{0, 1}+\alpha^2_{mr+1}a_{0, 1}^1)\right).&
\end{eqnarray}
Furthermore, we claim that the relation
\begin{eqnarray}\label{eq5}
&\nonumber\Scale[0.92]{(\alpha_{mp}^1, \alpha_{mp}^2, \beta_{mp}^1, \beta_{mp}^2)^\intercal=(-1)^{m(r-p)}\left(\prod_{i=mp}^{mr-1}(\mathcal{A}_{n,1}^{i,1})^{-1}\mathcal{A}_{n,1}^{i+1,0}\right)(\alpha_{mr}^1, \alpha_{mr}^2,\beta_{mr}^1, \beta_{mr}^2)^\intercal}&\\
&\Scale[0.92]{+(\mathcal{A}_{n, 1}^{mp, 1})^{-1}\!\left(\sum_{j=mp+1}^{mr}(-1)^{j-mp}\!\left(\prod_{i=mp+1}^{j-1}\!\mathcal{A}_{n, 1}^{i, 0}(\mathcal{A}_{n, 1}^{i, 1})^{-1}\right)\!(a^1_{j, n-j}, a^2_{j, n-j}, b^1_{j, n-j}, b^2_{j, n-j})^\intercal\right)}&
\end{eqnarray} leads to unique solutions for \(\alpha_{mp}^1\) and \(\alpha_{mr}^2\). The argument is as follows. Given the values \(\alpha_{mp}^2\) and \(\alpha_{mr}^1\) in equations \eqref{219}-\eqref{220}, the coefficients \(\alpha_{mp}^1\) and \(\alpha_{mr}^2\) are derived from the first and second rows of equations \eqref{eq5} while the coefficients \(\beta_{mp}^1\) and \(\beta_{mp}^2\) results from the third and fourth rows of \eqref{eq5} by taking \(\beta_{mr}^1=\beta_{mr}^2:=0.\) Finally, the coefficients \((\alpha_{j}^1, \alpha_{j}^2,\beta_{j}^1, \beta_{j}^2)\) for \(mp+1\leq j\leq mr-1\) are computed by the recursive relation \eqref{AscRec}. Let
\begin{equation}\label{NullSpace1}
\mathcal{A}_{n,1}(x_0, x_1, \cdots, x_{4n-1})^\intercal=0.
\end{equation} Then, we have
\begin{eqnarray*}
x_0=x_1=\cdots=x_{4mp-1}=x_{4mp+1}=0,\quad x_{4n-1}=x_{4n-2}=\cdots=x_{4mr+4}=x_{4mr}=0,\\
\mathcal{A}_{n}^{i,1}(x_{4i}, x_{4i+1}, x_{4i+2}, x_{4i+3})^\intercal+\mathcal{A}_{n}^{i+1,0}(x_{4(i+1)}, x_{4(i+1)+1}, x_{4(i+1)+2}, x_{4(i+1)+3})^\intercal=0
\end{eqnarray*} for \(mp\leq i\leq mr-1.\) Hence,
\begin{eqnarray}\label{eq6}
&\hspace{-1 cm}\Scale[0.95]{(x_{4mp}, x_{4mp+1},x_{4mp+2}, x_{4mp+3})^\intercal=(-1)^{m(r-p)}\left(\prod_{i=mp}^{mr-1}(\mathcal{A}_{n,1}^{i,1})^{-1}\!\mathcal{A}_{n,1}^{i+1,0}\right)\!(x_{4mr}, x_{4mr+1},x_{4mr+2}, x_{4mr+3})^\intercal.}&
\end{eqnarray}
By equations \eqref{BAn1} and \eqref{BAn0}, each of the matrices \(\mathcal{A}_{n,1}^{i,1}\) and  \(\mathcal{A}_{n,1}^{i,0}\) for \(0\leq i\leq n-1, i\neq mp, mr\)  has a structure in the form $\Scale[0.75]{\begin{pmatrix}A & \0 \\ B & cI_2\end{pmatrix}}$ where \(c\neq 0\) is a constant and \(A\) is an invertible matrix. Hence, their product  \(\prod_{i=mp}^{mr-1}(\mathcal{A}_{n,1}^{i,1})^{-1}\mathcal{A}_{n,1}^{i+1,0}\) has the same structure, \ie the format $\Scale[0.75]{\begin{pmatrix}A & \0 \\ B & cI_2\end{pmatrix}}$ with \(c\) as a nonzero constant, \(A\) as a \(2\times 2\) invertible matrix. Now we claim that the element on \((2,2)\)-th entry of \(A\) is non-zero.
On the contrary suppose that the element in the \((2,2)\)-th entry of the product matrix \(\Scale[0.9]{\prod_{i=mp}^{mr-1}(\mathcal{A}_{n,1}^{i,1})^{-1}\mathcal{A}_{n,1}^{i+1,0}}\) is zero. The second row in equations \eqref{eq6} leads to \(x_{4mp+1}=x_{4mr}=0\) while the first row causes to \(x_{4mp}=\gamma_1x_{4mr+1}+\gamma_2\) for a nonzero \(\gamma_1.\) The third and forth equations in \eqref{eq6} gives rise to \((x_{4mp+2}, x_{4mp+3})=(\gamma_3x_{4mr+3}+\gamma_4,\gamma_3x_{4mr+3}+\gamma_5 )\) for a nonzero \(\gamma_3.\) Thereby, the null space \eqref{NullSpace1} depends to the variables \(x_{4mr+1},\) \(x_{4mr+2},\) and \(x_{4mr+3}.\) This implies that \(\Null(\mathcal{A}_{n,1})=3.\) Therefore, \(\rank\, \mathcal{A}_{n,1}= 4n-3\) which is a contradiction.

Now consider the matrix \(\mathcal{A}_{n+1, 1}\) for \(n=m(p+q)+1=m(r+s)+1.\) By \eqref{eq3}, elements in \((2, 2)\)-entry of matrix \(\mathcal{A}_{n+1, 1}^{mp, 0}\) and \((1, 1)\)-entry of \(\mathcal{A}_{n+1, 1}^{mr+1, 1}\) are zero. Since \(p\leq r,\) \(mp<mr+1\) for every \(m\geq 1.\) This gives rise to \(\rank\, \mathcal{A}_{n+1, 1}= 4n+3.\) Hence, \(4n+3\)-number of \(n+1\)-grade terms of \(v^{(1)}\) are simplified in the second level normalization step. For \(0\leq j\leq mp-1\) and \(mp+1\leq j\leq n,\) we use the relation \eqref{AscRec}. Further, let \(\alpha_{mp}^2:=0,\)
\begin{eqnarray}\label{eq7}
&\nonumber \alpha_{mp}^1=\frac{a_{mp, mq+2}^1-2\alpha_{mp-1}^1((mp-2)a_{1, 0}^1+(n-mp+1)a_{1, 0}^2)}{2a_{0, 1}^2},&\\
&(\beta_{mp}^1, \beta_{mp}^2)=\frac{(b_{mp, mq+2}^1, b_{mp, mq+2}^2)}{2a_{0, 1}^2}+\frac{((1-mp)a_{1, 0}^1+(mp-1-n)a_{1, 0}^2)(\beta_{mp-1}^1, \beta_{mp-1}^2)+\alpha^1_{mp-1}(b_{1, 0}^1, b_{1, 0}^2)}{a_{0, 1}^2}.   &
\end{eqnarray} We here denote a vector \(V\) divided by a real number, say \(a_{0, 1}^2,\) to denote \(\frac{1}{a_{0, 1}^2}V\) for briefness. This  will be frequently used in this paper. Hence, terms \(\mathscr{P}^k_{n+1, 0}, \mathscr{R}^k_{n+1, 0}\) for \(k=1, 2\) and \(\mathscr{P}^2_{mp, n+1-mp}\) are not necessarily simplified in this level.

Case (ii): We have \(p>r.\) Since \(\ker \mathcal{A}_{n, 1}=\{0\},\) \(4n\)-number of \(n\)-grade terms are simplified. The relation \eqref{AscRec} is then applied for \(0\leq i\leq mp-1\) while the relation \eqref{DecRec} is implemented for \(n\geq i\geq mp+2\). Thereby, terms \(\mathscr{P}^k_{j, n-j}\) and \(\mathscr{R}^k_{j, n-j}\) for \(k=1,2, 0\leq j\leq n, j\neq mp\) are simplified.

For the matrix \(\mathcal{A}_{n+1, 1},\) the second row of \(\mathcal{A}_{n+1, 1}^{mp, 0}\) and the first row of \(\mathcal{A}_{n+1, 1}^{mr+1, 1}\) are both zero. Now the inequality \(p>r\) implies that \(mp\geq mr+1.\) Hence, we consider two subcases: (1) \(mp=mr+1\) and (2) \(mp>mr+1.\)

The subcase (1) infers that \(m=1\) and \(p=r+1\). Here, \(\rank(\mathcal{A}_{n+1, 1})=4n+2\) and the coefficients are assigned similar to the case \(p\leq r\) in (ii). Therefore, the terms \(\mathscr{P}^2_{p, q+2}, \mathscr{P}^1_{p+1, q+1}\) and \(\mathscr{R}^k_{p+1, q+1}\) for \(k=1, 2\) are not necessarily simplified.

The subcase (2) holds when \(m\geq 2\) or \(p>r+1\). Due to the \(\Null(\mathcal{A}_{n+1, 1})=0,\) \(4n+4\)-number of terms are specifically eliminated from the first level normal form. Coefficients are here assigned similarly. Indeed for \(0\leq i\leq mp-1,\) we assign the relation \eqref{AscRec}, and  for \(mp+2\leq i\leq n+2,\) we use the relation \eqref{DecRec} where \((\alpha^1_{n+1}, \alpha^2_{n+1}, \beta^1_{n+1}, \beta^2_{n+1})=(0, 0, 0, 0).\) Therefore, terms \(\mathscr{P}^k_{mq+2}\) and \(\mathscr{R}^k_{mp, mq+2}\) for \(k=1, 2\) may remain in the second level normal form.

Now assume that \(p+q>2,\) \(n=m(p+q)+j\) for a \(j=3,\cdots, p+q\) and \(m\in \mathbb{Z}_{\geq 0}.\) Due to the invertibility of matrices \(\mathcal{A}_{n, 1}^{j, 0}\) for \(0\leq j\leq n-1,\) all columns of \(\mathcal{A}_{n, 1}\) are linearly independent. As a result  \(\rank\, \mathcal{A}_{n, 2}=4n,\) and terms \( \mathscr{P}^k_{j, n-j}\) and \(\mathscr{R}^k_{j, n-j}\) for \(k=1, 2, 0\leq j\leq n-1\) are simplified. Therefore, the recursive relation \eqref{AscRec} is assigned for the coefficients, and the second level normal form of \(v^{(1)}\) is given by \eqref{SecondLevel1}.
\end{proof}
\noindent Now we introduce a partition for the natural numbers given by
\begin{equation*}
\begin{split}
&P_1\!:=\!\{m \lcm(p\!+\!q, r\!+\!s)\!+\!1|\, m\in \mathbb{N}\},  \\
&\Scale[0.88]{ P_3\!:=\!\{m(p\!+\!q)\!+\!1|\,\exists m,m^{\prime}\in \mathbb{N}, m(p\!+\!q)\!-\!m^{\prime}(r\!+\!s)\!=\!1\}},
\end{split}\;
\begin{split}
& P_2\!:=\!\{m \lcm(p\!+\!q, r\!+\!s)+2|\,m\in\mathbb{N}\}, \\
&\Scale[0.88]{ P_4\!:=\!\{m^{\prime}(r\!+\!s)\!+\!1|\,\exists m,m^{\prime}\in \mathbb{N}, m^{\prime}(r\!+\!s)\!-\!m(p\!+\!q)\!=\!1\}},
\end{split}\end{equation*} \(P_5\!:=\!\{m(p\!+\!q)\!+\!1|m\in \mathbb{N}\}\setminus P_1\cup P_3,\) \( P_6\!:=\!\{m(p\!+\!q)\!+\!2|\,m\in \mathbb{N}\}\setminus P_2\cup P_4,\) and \(P_7:=\mathbb{N}\setminus\cup_{i=1}^6P_i.\) Also define \(\mathbb{N}_m:=\{1, 2, \cdots, m\}.\) In order to only deal with \(2N+1\)-degree truncated normal forms for some sufficiently large \(N\), from now on and for our convenience, we assume that all sums have a finite nonzero terms upto degree \(2N+1\).

\begin{lem}\label{Lem2.4} Let conditions \eqref{Conditions} hold, \(r+s\neq p+q\) and \(\gcd(r+s, p+q)=1\). Therefore,
the second level normal form of \(v^{(1)}\) is given by
\begin{eqnarray}
&\nonumber\Theta+\Sum{k\in\mathbb{N}_2}{}{(a^k_{0, 1}\mathscr{P}^k_{0, 1}+b^k_{0, 1}\mathscr{R}^k_{1, 0})+a^2_{0, 2}\mathscr{P}^2_{0, 2}}+\Sum{n\in P_7, i\in \mathbb{N}_2}{} \Scale[0.9]{\big(a^k_{n,0}\mathscr{P}^k_{n, 0}+b^k_{n, 0}\mathscr{R}^k_{n, 0}\big)},&\\
&\Sum{k\in \mathbb{N}_6, i\in\mathbb{N}_2}{}\Sum{n\in P_k,j\in Q_k}{} \Scale[0.9]{(a^k_{j, n-j}\mathscr{P}^k_{j,n-j}+b^k_{j,n-j}\mathscr{R}^k_{j,n-j})}&\label{SecondLevel2_2}
\end{eqnarray}
where \({Q_5:=\left\{\frac{(n-1)p}{p+q}\right\}}\) and \({Q_6:=\left\{\frac{(n-2)p}{p+q}\right\},}\)
\begin{itemize}
\item \({Q_1:=\left\{\frac{(n-1)p}{p+q}, \frac{(n-1)r}{r+s}+1\right\}}\) for \(\frac{p}{p+q}\leq \frac{r}{r+s},\) and  \({Q_1:=\left\{\frac{(n-1)p}{p+q}\right\}}\) when
\(\frac{p}{p+q}> \frac{r}{r+s}.\)
\item \({Q_2=\left\{m, n\right\},}\, a^1_{m, n-k}=b^k_{m, n-m}=0\) for \({k=1,2, m=\frac{(n-2)p}{p+q}}\)  when \({m\leq \frac{(n-2)r}{r+s}+1}.\) Otherwise,
\({Q_2:=\left\{k\right\}}.\)
\item \({Q_3:=\left\{0, m+1\right\}, a^2_{m, n-m}=b^k_{m, n-m}=0}\) for \({k=1, 2, m=\frac{(n-2)r}{r+s}+1}\) if \({\frac{(n-1)p}{p+q}\leq m}.\) For the remaining cases, let
\({Q_3:=\left\{\frac{(n-1)p}{p+q}\right\}}.\)
\item \(Q_4:=\left\{m, n\right\},\) \(a^1_{m, n-m}=b^k_{m, n-m}=0\) for \(k=1, 2,m=\frac{(n-2)p}{p+q}\) if \(m\leq \frac{(n-1)r}{r+s}.\)
Otherwise, we define \(Q_4:=\left\{ k\right\}.\)
\end{itemize}
\end{lem}
\begin{proof}
Since \(p+q\) and \(r+s\) are coprime, \(\lcm(p+q, r+s)=(p+q)(r+s).\) Let \(n\in P_1.\) Thus, \(n=m(r+s)(p+q)+1\) for a \(m\in \mathbb{N}.\) Further, the first, third and fourth columns of matrix \(\mathcal{A}_{n, 1}^{m(r+s)p, 0}\) and the second, third and fourth columns of the matrix \(\mathcal{A}_{n, 1}^{m(p+q)r, 1}\) are zero vectors. Thereby, we apply an approach similar to the proof of Lemma \ref{lem1}. Indeed, we consider two cases:
\bes
(I) \quad \frac{p}{p+q}\leq \frac{r}{r+s} \qquad\hbox{ and } \qquad (II)\quad \frac{p}{p+q}>\frac{r}{r+s}.
\ees
For the case (I), terms \(\mathscr{P}^2_{j, n-j}, \mathscr{R}^1_{j, n-j}, \mathscr{R}^2_{j, n-j}\) for \(j\in \{m(r+s)p, m(p+q)r+1\}\) may still remain in the second level normal form. For the case (II), the polynomial vector fields \(\mathscr{P}^k_{j, n-j}\) and \(\mathscr{R}^k_{j, n-j},\) for \(k=1, 2,\) \(0\leq j\leq n, j\neq m(r+s)p,\) can be eliminated from the first level normal form. The assignments of coefficients for \(0\leq i\leq m(r+s)p-1\) follows \eqref{AscRec} and for \(m(r+s)p+2\leq i\leq n+1,\) we apply \eqref{DecRec}.

Let \(n\in P_2.\) Thus, the element in \((2, 2)\)-entry of matrix \(\mathcal{A}_{n, 1}^{m(r+s)p, 0}\) and \((1, 1)\)-entry of \(\mathcal{A}_{n, 1}^{m(p+q)r+1, 1}\) are zero. When \(m(r+s)p\leq m(p+q)r+1,\) all \(n\)-grade homogenous terms \(\mathscr{P}^k_{j, n-j}\) and \( \mathscr{R}^k_{j, n-j}\) for \(k=1, 2\), \(0\leq j\leq n-1,\) (with only one exception) are simplified. The exception here is the polynomial vector field \(\mathscr{P}^2_{m(r+s)p, m(r+s)q+2}\), where it may still remain in the normalized system. If \(m(r+s)p> m(p+q)r+1,\) the Eulerian \(\mathscr{P}^k_{j, n-j}\) and rotating vector fields \(\mathscr{R}^k_{j, n-j}\) for \(k=1, 2, j=m(r+s)p,\) cannot be necessarily eliminated in the second level normal form step.

Now we first claim that \(P_3\neq \O.\) By the assumption, \(p+q\) and \(r+s\) are coprime. By {\it Bezout's lemma,} there exist nonzero integers \(n_1\) and \(n_2\) so that \(n_1(p+q)+n_2(r+s)=1.\) The Bezout coefficients are expressed as \((n_1-\ell(r+s), n_2+\ell(p+q))\) for \(\ell\in \mathbb{Z}\). Since \(p+q\) and \(r+s\) are positive, \(n_1n_2<0.\) When \(n_1>0\) and \(n_2<0\), \(n_1(p+q)+1\in P_3.\) Otherwise, for all \(k^{\prime}\leq \min\left\{\left\lfloor \frac{n_1}{r+s}\right\rfloor, \left\lfloor \frac{-n_2}{p+q}\right\rfloor\right\}\), \((n_1-\ell^{\prime}(r+s))(p+q)+1\in P_3.\) Therefore, for \(n\in P_3,\) there exist positive integers \(m, m^{\prime}\) so that \(n=m(p+q)+1=m^{\prime}(r+s)+2.\) Recall that the  first, third and fourth columns of \(\mathcal{A}_{n, 1}^{mp, 0}\) are zero vectors. The \((1,1)\)-entry of matrix \(\mathcal{A}_{n, 1}^{m^{\prime}r+1, 1}\) is also zero. When \(mp\leq m^{\prime}r+1,\) \(\Null(\mathcal{A}_{n, 1})=1.\) Hence, \(4n-1\)-number of \(n\)-grade terms can be simplified. The coefficients follow equation \eqref{DecRec} for \(m^{\prime}r+3\leq i\leq n\) and \(1\leq i\leq m^{\prime}r+1.\) Hence, we let \(\alpha_{m^{\prime}r+1}^1:=0,\) and obtain
\begin{eqnarray*}
&&\alpha_{m^{\prime}r+1}^2=\frac{a_{m^{\prime}r+2}^2-(a_{0, 1}^1(m^{\prime}r+2)+a^2_{0, 1}(m^{\prime}s-2))\alpha^2_{m^{\prime}r+2}}{a_{1, 0}^1(m^{\prime}r+1)+a_{1, 0}^2m^{\prime}s},\\
&&\beta_{m^{\prime}r+1}^1=\frac{b_{m^{\prime}r+2}^1-(a_{0, 1}^1(m^{\prime}r+2)+a^2_{0, 1}(m^{\prime}s-1))\beta^1_{m^{\prime}r+2}+b^1_{0, 1}\alpha_{m^{\prime}r+2}^2}{a_{1, 0}^1(m^{\prime}r+1)+a_{1, 0}^2m^{\prime}s},\\
&&\beta_{m^{\prime}r+1}^2=\frac{b_{m^{\prime}r+2}^2-(a_{0, 1}^1(m^{\prime}r+2)+a^2_{0, 1}(m^{\prime}s-1))\beta^2_{m^{\prime}r+2}+b^2_{0, 1}\alpha_{m^{\prime}r+2}^2}{a_{1, 0}^1(m^{\prime}r+1)+a_{1, 0}^2m^{\prime}s}.
\end{eqnarray*}
Therefore, terms \(\mathscr{P}^k_{0, n}, \mathscr{R}^k_{0, n}\) for \(k=1, 2\) and \(\mathscr{P}^1_{m^{\prime}r+2, m^{\prime}s-1}\) are not necessarily simplified. When \(mp> m^{\prime}r+1,\) according to \(\Null(\mathcal{A}_{n, 1})=0\) and invertibility of matrices \(\mathcal{A}_{n, 1}^{j, 0}\) and \(\mathcal{A}_{n ,1}^{i,1}\) for \(0\leq j\leq mp-1\) and \(mp\leq i\leq n-1,\)  coefficients are determined by the recursive relation \eqref{AscRec} when \(0 \leq i\leq mp-1\) and  the relation \eqref{DecRec} for \(mp+1\leq i\leq n\). Hence, terms \(\mathscr{P}^k_{j, n-j}, \mathscr{R}^k_{j, n-j}\) for \(k=1, 2\) and \(0\leq j\leq n, j\neq mp,\) are removed from the second level normal form system. Similar to the above, \(P_4\) is nonempty. Let \(n\in P_4.\) The second, third and fourth columns of \(\mathcal{A}_{n, 1}^{m^{\prime}r, 1}\) are zero vectors. Also, \((2, 2)\)-entry of \(\mathcal{A}_{n, 1}^{mp, 0}\) is zero. When \(mp\leq m^{\prime}r,\) row reduced echelon form of transpose \(\mathcal{A}_{n, 1}\) has a one dimensional null space. Hence, rank of \(\mathcal{A}_{n, 1}\) is \(4n-1\). Due to the invertibility of \(\mathcal{A}_{n, 1}^{i, 0}\) for \(0\leq i\leq n-1, i\neq mp,\) the coefficients are assigned by the relation \eqref{AscRec} where \(\alpha_{mp}^2:=0\) and the coefficients \(\alpha_{mp}^1, \beta_{mp}^1, \beta^2_{mp}\) are defined by equations \eqref{eq7}. Accordingly, terms \(\mathscr{P}^k_{n, 0}\) and \(\mathscr{R}^k_{n, 0}\) for \(k=1, 2\) and \(\mathscr{P}^2_{mp, n-mp}\) are not necessarily simplified in the second level step. If \(mp> m^{\prime}r,\) \(\rank(\mathcal{A}_{n, 1})=4n.\) The coefficient assignments is the same as the cases for \(n\in P_3.\) Thereby, for \(k=1, 2,\) \(0\leq j\leq n, j\neq mp,\) terms \(\mathscr{P}^k_{j, n-j}\) and \(\mathscr{R}^k_{j, n-j}\) can be eliminated.

If \(n\in P_5,\) the first, third and fourth columns of matrix \(\mathcal{A}_{n, 1}^{mp, 0}\) are zero vectors. Due to linear independence of the column space of \(\mathcal{A}_{n, 1},\) rank of this matrix is \(4n.\) For \(0\leq i\leq mp-1,\) coefficients are determined by the recursive relation \eqref{AscRec} while for \(mp+2\leq j\leq n\), the relation \eqref{DecRec} is applied. Therefore, all terms \(\mathscr{P}^k_{j, n-j}, \mathscr{R}^k_{j, n-j}\) for \(k=1, 2\) and \(0\leq j\leq n, j\neq mp,\) can be normalized.

Take \(n\in P_6.\) Then, \((2, 2)\)-entry of \(\mathcal{A}_{n, 1}^{mp, 0}\) is zero and rank of \(\mathcal{A}_{n, 1}\) is \(4n\). The coefficients are determined similar to \(P_3\)-cases. Hence, terms \(\mathscr{P}^k_{mp, n-mp}\) and \(\mathscr{R}^k_{mp, n-mp}\) for \(k=1, 2\) may remain in the second level step. For \(n\in P_7\), we have
\bes P_7=\{m(p+q)+j|j\in \mathbb{N}, 3\leq j\leq p+q\}.\ees Since all matrices \(\mathcal{A}_{n, 1}^{i, 0}\) and \(\mathcal{A}_{n, 1}^{i, 1}\) for \(0\leq i\leq n-1\) are invertible, \(\rank(\mathcal{A}_{n, 1})=4n.\) Hence all \(n\)-grade homogenous terms \(\mathscr{P}^k_{j, n-j}\) and \( \mathscr{R}^k_{j, n-j}\) for \(k=1, 2\), \(0\leq j\leq n-1,\) are removed from the second level normal form.
\end{proof}

\begin{lem} Assume that conditions \eqref{Conditions} are satisfied. If \(r+s\neq p+q\) and \(\gcd(r+s, p+q)\!>\!1,\) the second level normal form of \(v^{(1)}\) is given by
\begin{eqnarray}\nonumber
&\Theta+ \sum_{k\in\mathbb{N}_2}{(a^k_{0, 1}\mathscr{P}^k_{0, 1}+b^k_{0, 1}\mathscr{R}^k_{1, 0})+a^2_{0, 2}\mathscr{P}^2_{0, 2}}+\sum_{n\in P_7, i\in \mathbb{N}_2} {\big(a^k_{n,0}\mathscr{P}^k_{n, 0}+ b^k_{n, 0}\mathscr{R}^k_{n, 0}\big)}&\\
&+ \sum_{m\in \{1, 2, 5, 6\}, k\in\mathbb{N}_2}\sum_{n\in P_m,j\in Q_m} {(a^k_{j, n-j}\mathscr{P}^k_{j,n-j}+b^k_{j,n-j}\mathscr{R}^k_{j,n-j})},&\label{SecondLevel2_3}
\end{eqnarray}
where \(Q_i\) for \(i=1, 2, 5, 6, 7\) are defined in Lemma  \ref{Lem2.4} and \(Q_3=Q_4=P_3=P_4=\emptyset.\)
\end{lem}
\begin{proof}
Due to \(\gcd(p+q, r+s)\neq 1,\) \(P_3\) and \(P_4\) are empty. The argument for the cases of \(n\in P_5, n\in P_6\) and \(n\in P_7\) are the same as in Lemma \ref{Lem2.4}. Hence, it suffices to study the cases for \(n\in P_1\) and \(n\in P_2\). Let \(d:=\lcm(p+q, r+s)\) and \(n:=md+1\in P_1\) for a \(m\in \mathbb{N}\). Since \(\frac{dmp}{p+q}\in \mathbb{N},\) we have
\begin{eqnarray*}
&\frac{dmp}{p+q}a^1_{0, 1}+\left(n-1-\frac{dmp}{p+q}\right)a^2_{0, 1}=\frac{dmp}{p+q}a^1_{0, 1}+\frac{dmq}{p+q}a^2_{0, 1}=0.&
\end{eqnarray*} Hence, the first, third and fourth columns of the matrix \(\mathcal{A}_{n ,1}^{\frac{dmp}{p+q}}\) are zero vectors. Via a similar argument, the second, third and fourth columns of the matrix \(\mathcal{A}_{n ,1}^{\frac{dmr}{r+s}}\) are zero. For the case \(\frac{p}{p+q}\leq \frac{r}{r+s},\) as of the argument in \(P_1\)-case of Lemma \ref{Lem2.4}, we can eliminate \(\mathscr{P}^k_{j, n-j}, \mathscr{R}^k_{j, n-j}\) when \(k=1, 2, 0\leq j\leq n,\) for \(j\neq \frac{dmp}{p+q}\) and \(j\neq \frac{dmr}{r+s}+1.\) Furthermore, \(\mathscr{P}^1_{j, n-j}\) can be normalized when \(j=\frac{dmp}{p+q}\) and \(j=\frac{dmr}{r+s}+1.\) If \(\frac{p}{p+q}>\frac{r}{r+s},\) the vector fields \(\Scale[0.85]{\mathscr{P}^k_{\frac{dmp}{p+q}, n-\frac{dmp}{p+q}}, \mathscr{R}^k_{\frac{dmp}{p+q}, n-\frac{dmp}{p+q}}}\) for \(k=1,2\) may remain in the normal form system.

Now assume that \(n=md+2\in P_2\) for a \(m\in \mathbb{N}.\) When \(\frac{dmp}{p+q}\leq \frac{dmr}{r+s}+1\), terms \(\mathscr{P}^k_{n, 0}, \mathscr{R}^k_{n, 0}\) for \(k=1,2\) and \(\Scale[0.85]{\mathscr{P}^2_{\frac{dmp}{p+q}, n-\frac{dmp}{p+q}}}\) may stay in the normal form system. For \(\frac{dmp}{p+q}>\frac{dmr}{r+s}+1,\) all vector fields \(\mathscr{P}^k_{j, n-j},\) \(\mathscr{R}^k_{j,n-j}\) for \(k=1,2\) and \(0\leq j\leq n, j\neq \frac{dmp}{p+q},\) are simplified. Therefore, the second level normal form is expressed by the vector field \eqref{SecondLevel2_2}. Due to \(P_3=P_4=\emptyset,\) the summations \(\sum_{n\in P_3}{},\) \(\sum_{j\in Q_3}{}\) \(\sum_{n\in P_4}{},\) and \(\sum_{j\in Q_4}{}\) are all zero for both cases of \(\frac{dmp}{p+q}\leq \frac{dmr}{r+s}+1\) and \(\frac{dmp}{p+q}>\frac{dmr}{r+s}+1\). This gives rise to our claim.
\end{proof}
\begin{cor}\cite[Corollary 3.5]{GazorHamziShoghi}\label{lem2}
Let conditions \eqref{Conditions} hold and \(\frac{s}{r}=\frac{q}{p}\). Then, the second level normal form of \(v^{(1)}\) in equation \eqref{pre-classical} is expressed by
\begin{eqnarray*}
&\Theta\!+\!a^1_{0, 1}\mathscr{P}^1_{0, 1}\!+\!a^1_{1, 0}\mathscr{P}^1_{1, 0}\!+\!\Sum{m\in \mathbb{Z}_{\geq 0}, j\in\{0,1\}}{}\left(a^2_{mp+j, mq+1-j}\mathscr{P}^2_{mp+j, mq+1-j}\!+\!\Sum{k\in\mathbb{N}_2}{} b^k_{mp+j, mq+1-j }\mathscr{R}^k_{mp+j, mq+1-j}\right)&\\ [-4.5pt]
&+\Sum{m\in \mathbb{Z}_{\geq 0}}{}\left( a^2_{mp, mq+2}\mathscr{P}^2_{mp, mq+2}\!+\!\Sum{k\in \mathbb{N}_2}{} \big(a^k_{mp+2, mq}\mathscr{P}^k_{mp+2, mq}+ b^k_{mp+2, mq}\mathscr{R}^k_{mp+2, mq}\big)\right)&\\ [-4.5pt]
&+\Sum{m\in \mathbb{Z}_{\geq 0}}{} \Sum{k\in \mathbb{N}_2}{}\Sum{j=3}{p+q} \left(a^k_{m(p+q)+j, 0}\mathscr{P}^k_{m(p+q)+j, 0}\!+\! b^k_{m(p+q)+j, 0}\mathscr{R}^k_{m(p+q)+j, 0}\right).
\end{eqnarray*}
\end{cor}

\subsection{Case II: either \(a_{0,1}^1a_{0, 1}^2>0\), \(a_{1, 0}^1a_{1, 0}^2>0,\) \(\frac{a_{0, 1}^1}{a_{0, 1}^2}\in \mathbb{Q}^c\), or \(\frac{a_{1, 0}^1}{a_{1, 0}^2}\in \mathbb{Q}^c.\)}

\begin{lem}\cite[Lemma 3.6]{GazorHamziShoghi}\label{lem3} Assume \({a_{0,1}^1}{a_{0,1}^2}>0\), \(\frac{a_{0,1}^2}{a_{0,1}^1}\notin\mathbb{N}, \) and consider the vector field
\begin{eqnarray}
\textstyle v^{(2)}=\Theta+{}\Sum{k\in \mathbb{N}_2}{}\left( a^k_{0, 1}\mathscr{P}^k_{0, 1}+b^k_{0, 1}\mathscr{R}^k_{0, 1}\right)+a^2_{0,2}\mathscr{P}_{0,2}^2+a^2_{1, 1}\mathscr{P}^2_{1, 1}+ \Sum{j\in \mathbb{N}}{}\Sum{k\in \mathbb{N}_2}{}\left( a^k_{j,0}\mathscr{P}^k_{j,0} +b^k_{j,0}\mathscr{R}_{j,0}^k \right), \label{SecondLevel3}
\end{eqnarray}
The second level normal form of $v^{(1)}$ is given by \eqref{SecondLevel3} where \(a^2_{1, 1}=0\).
\end{lem}
\begin{lem}\label{LemR} Let \(\frac{a_{0,1}^2}{a_{0,1}^1}\in\mathbb{N}\) and \({a_{0,1}^2}\neq{a_{0,1}^1}.\) The second level normal form of $v^{(1)}$ is given by \eqref{SecondLevel3} where \(a^2_{1, 1}=0\) and the corresponding term with \(a^2_{n-1, 1}\) is not necessarily simplified. When \({a_{0,1}^2}={a_{0,1}^1},\) the second level is expressed by \eqref{SecondLevel3} in which
\begin{enumerate}
  \item\label{c2} \(a^2_{2, 0}=0,\) if  \(a^1_{1, 0}\neq a^2_{1, 0}\).
  \item\label{c3} \(b^1_{2, 0}= 0,\) when  \(a^1_{1, 0}= a^2_{1, 0}\) and \(|b^1_{1, 0}|+|a^1_{1, 0}|\neq 0.\)
  \item\label{c4} \(b^2_{2, 0}= 0,\) for  \(a^1_{1, 0}= a^2_{1, 0},\) \(b^1_{1, 0}=a^1_{1, 0}= 0\) and \(b^2_{1, 0}\neq 0.\)
\end{enumerate}
\end{lem}
\begin{proof} Assume \(\frac{a_{0,1}^2}{a_{0,1}^1}\in\mathbb{N}\) and \({a_{0,1}^2}\neq{a_{0,1}^1}.\) Since the second row of the matrix \(\mathcal{A}_{n, 1}^{n-1, 0}\) is a zero vector, the term \(\mathscr{P}^2_{n-1, 1}\) is not necessarily simplified. For the cases when \(a^1_{1, 0}\neq 0,\) the term associated with \(a^2_{n, 0}\) can be instead normalized. The rest of the proof is obtained through an approach similar to the proof of Lemma \ref{lem3}. Now let \({a_{0,1}^2}={a_{0,1}^1}.\)
 For claim \ref{c2}, the first and sixth rows of the matrix \(\mathcal{A}_{2,1}\) are linearly dependent and the second row of \(\mathcal{A}_{2, 1}^{1, 0}\) is a zero vector. Hence, term \(\mathscr{P}^2_{1, 1}\) cannot necessarily be eliminated in the second level step. Due to \(a^1_{1, 0}\neq a^2_{1, 0},\) we
interchange the fifth row with the eighth row in \(\mathcal{A}_{2,1}\) to obtain a modified matrix \(\widehat{\mathcal{A}}_{2,1},\) where its sub-matrix \(\widehat{\mathcal{{A}}_{2,1}^{1, 0}}\) is invertible. Next, the coefficients \((\alpha^1_{0}, \alpha^2_{0}, \beta^1_{0}, \beta^2_{0})\) are derived by relation \eqref{n=2Coefs1} and coefficients \((\alpha^1_{1}, \alpha^2_{1}, \beta^1_{1}, \beta^2_{1})\) are computed from equation \eqref{n=2Coefs2} where matrices \(\mathcal{{A}}_{2,1}^{1, 0}\) and \(\mathcal{{A}}_{2,1}^{0, 1}\) are replaced with \(\widehat{\mathcal{{A}}_{2,1}^{1, 0}}\) and \(\widehat{\mathcal{{A}}_{2,1}^{0, 1}},\) respectively. We further need to replace \(a^2_{1, 1}\) with \(a^2_{2, 0}.\) Therefore, terms \(\mathscr{P}^1_{j, 2-j}, \mathscr{R}^1_{j, 2-j}, \mathscr{R}^2_{j, 2-j}\) and \(\mathscr{P}^2_{2, 0}\) for \(j=0, 1\) are simplified in our second level normal form step.

In claim \ref{c3}, we have \({a_{0,1}^1}= {a_{0,1}^2}\) and \(a^1_{1, 0}= a^2_{1, 0}.\) In this case, the eighth and fifth rows of \(\mathcal{A}_{2,1}\) are lineally dependent. Then, term \(\mathscr{P}^2_{2, 0}\) may not be eliminated in the second level step. Since \(|b^1_{1, 0}|+|a^1_{1, 0}|\neq 0,\) the eleventh row of \(\mathcal{A}_{2,1}\) is not a zero vector and if we interchange it with sixth row, the modified matrix \(\widehat{\mathcal{{A}}_{2,1}^{1, 0}}\) becomes invertible. We, of course, need to replace \(b^1_{2, 0}\) with \(a^2_{1, 1}.\) Terms \(\mathscr{P}^2_{0, 2},\mathscr{P}^2_{1, 1}, \mathscr{P}^1_{2, 0}, \mathscr{P}^2_{2, 0}\) and \(\mathscr{R}^2_{2, 0}\) may not be normalized from the first level normal form.

For claim \ref{c4}, we have \(b^1_{1, 0}=a^1_{1, 0}= 0\) and thus, the eleventh row of \(\mathcal{{A}}_{2,1}\) is a zero vector. Due to \(b^2_{1, 0}\neq 0,\) the twelfth row of \(\mathcal{{A}}_{2,1}\) is non-zero. The coefficients are assigned similar to claim \ref{c3} while the twelfth row is replaced with fifth row of \(\mathcal{{A}}_{2,1}\) and \(a^1_{1, 1}\) is interchanged with \(b^2_{2, 0}.\) Hence, in claims \ref{c2}, \ref{c3} and \ref{c4}, terms \(\mathscr{P}^1_{j, 2-j}, \mathscr{R}^1_{j, 2-j}, \mathscr{R}^2_{j, 2-j}\) and \(\mathscr{R}^2_{2, 0}\) for \(j=0, 1\) are normalized in the second level normal form step.
\end{proof}
\begin{lem}\label{lem3-1}  When \({a_{1,0}^1}{a_{1,0}^2}>0\) and \(\frac{a_{1,0}^1}{a_{1,0}^2}\notin \mathbb{N},\) there are near identity polynomial changes of  state variables such that they send $v^{(1)}$ into its second level normal form given by
\begin{eqnarray}
\textstyle v^{(2)}=\Theta+\Sum{k\in\mathbb{N}_2}{}\left( a^k_{1, 0}\mathscr{P}^k_{1, 0}+b^k_{1, 0}\mathscr{R}^k_{1, 0}\right)+a^1_{2, 0}\mathscr{P}_{2, 0}^1+a^1_{1, 1}\mathscr{P}^1_{1, 1}+ \Sum{j\in \mathbb{N}}{}\Sum{k\in\mathbb{N}_2}{}\left( a^k_{0, j}\mathscr{P}^k_{0, j} +b^k_{0, j}\mathscr{R}_{0, j}^k \right) \label{SecondLevel3-1}
\end{eqnarray}
where \(a^2_{1, 1}=0.\)
\end{lem}
\begin{proof} When \(n=2\) and \({a_{1, 0}^1}\neq {a_{1, 0}^2},\) the coefficients are defined as \(\alpha_1^1=0,\) \(\alpha_1^2=\frac{a^2_{2, 0}}{2a_{1, 0}^1},\) \(\beta_1^1=\frac{b^1_{2, 0}}{2a_{1, 0}^1},\) \(\beta_1^2=\frac{b^2_{2, 0}}{2a_{1, 0}^1},\) and
\begin{eqnarray*}
&(\alpha^1_0, \alpha^2_0, \beta^1_0, \beta^2_0)^\intercal={(\mathcal{A}_{2 ,1}^{0,1})}^{-1}\left({(a^1_{1, 1}, a^2_{1, 1}, b^1_{1, 1}, b^2_{1, 1})}^\intercal-\mathcal{A}_{2 ,1}^{1, 0}{(\alpha_1^1, \alpha_1^2, \beta_1^1, \beta_1^2)}^\intercal\right).&
\end{eqnarray*} These simplify all terms \(\mathscr{P}^k_{j, 2-j}\) and \(\mathscr{R}^k_{j, 2-j}\) for \(j=1, 2, k=1, 2,\) with only one exception, that is, \(\mathscr{P}^1_{2, 0}\). Proofs for other claims when \(n=2\) are similar to Lemma \ref{lem3}, where we omit them for briefness.

For \(n>2\), matrices \(\mathcal{A}_{n,1}^{k, 1}\) for \(n\geq k\geq 1\) in \eqref{BAn1} are invertible. Then, the recursive relation \eqref{DecRec} for \(0\leq i\leq n-1\) is applied to normalize terms \(\mathscr{P}^1_{0, n}, \mathscr{P}^2_{0, n}, \mathscr{R}^1_{0, n} \) and \(\mathscr{R}^2_{0, n}.\)
\end{proof}
\begin{lem}\label{lem38}
  \begin{enumerate}
    \item Let \(\frac{a_{0, 1}^1}{a_{0, 1}^2} \in \mathbb{Q}^c.\) Then, the second level normal form of $v^{(1)}$ in equation \eqref{pre-classical} follows \eqref{SecondLevel3} where \(a^2_{1, 1}=0.\)
    \item When \(\frac{a_{1,0}^1}{a_{1,0}^2} \in \mathbb{Q}^c,\) the second  level normal form of $v^{(1)}$ is expressed by \eqref{SecondLevel3-1} where \(a^1_{1, 1}=0.\)
  \end{enumerate}
\end{lem}
\begin{proof}
  Assume that \(n=2.\) Since \(\frac{a_{0, 1}^1}{a_{0, 1}^2}\neq 1\) and \(\frac{a_{1,0}^1}{a_{1,0}^2}\neq 1,\) according to  Lemmas \ref{lem3} and \ref{lem3-1}, \(a^1_{1, 1}=a^2_{1, 1}=0.\) For \(n>2,\) the argument is similar to the proof of Lemma  \ref{lem3-1} and is omitted for briefness.
\end{proof}
\begin{exm}
Consider the first level normal form
\begin{eqnarray}\label{Example2}
&\nonumber v^{(1)}:=\mathscr{R}_{0, 0}^1+ \sqrt{2}\mathscr{R}_{0, 0}^2+
{\alpha}\mathscr{P}_{0, 1}^1+\beta\mathscr{P}_{0, 1}^2+\mathscr{P}_{1, 0}^1-\mathscr{P}_{1, 0}^2+{}\mathscr{P}_{1, 1}^1+ \mathscr{P}_{1, 1}^2+{}\mathscr{P}_{ 2, 0}^2+\mathscr{R}_{2, 0}^1+\frac{1}{6}\mathscr{P}_{2, 1}^1&\\
&+\frac{2}{3}\mathscr{P}_{2, 1}^2 +\mathscr{P}_{3, 0}^2+\mathscr{P}_{3, 1}^1+\mathscr{P}_{3, 1}^2+\mathscr{P}_{4, 0}^2+\mathscr{R}_{4, 0}^2+\mathscr{P}_{4, 1}^1+\mathscr{P}_{4, 1}^2+\mathscr{P}_{5, 0}^1+\mathscr{P}_{5, 0}^2.&
\end{eqnarray} and assume that  \(\alpha=1, \beta=2.\) Hence, \(a_{0, 1}^1=1, a_{0, 1}^2=2\) and \(a_{0, 1}^1a_{0, 1}^2>0.\) These conditions are matched with Lemma \ref{LemR} where \(\frac{a_{0, 1}^1}{a_{0, 1}^2}=\frac{1}{2}\). Following proof of Lemma \ref{lem3}, we consider three cases:
\begin{itemize}
  \item[1.] \emph{Normalizing 2-grade terms}: For computation of coefficients we have \(\alpha^k_{1, 0}=\beta^k_{1, 0}=0\) for \(k=1, 2\) and
  \begin{eqnarray*}
(\alpha^1_{1, 0}, \alpha^2_{1, 0}, \beta^1_{1, 0}, \beta^2_{1, 0})^\intercal=\Scale[0.8]{\begin{bmatrix}
                                                                   2 & -2 & 0 & 0 \\
                                                                   0  &-2 & 0 & 0 \\
                                                                   0  & 0 & 2& 0 \\
                                                                   0  & 0 & 0 & 2
                                                                  \end{bmatrix}^{-1} \begin{bmatrix}
                                                                                  1 \\
                                                                                  1 \\
                                                                                  0 \\
                                                                                  0
                                                                                \end{bmatrix}=\begin{bmatrix}
                                                                                                0 \\
                                                                                                -\frac{1}{2} \\
                                                                                                0 \\
                                                                                                0
                                                                                              \end{bmatrix}.}
  \end{eqnarray*}
  Therefore, the transformation generator for simplifying the vector field in grade 2 is given by \(X_1:=-\frac{1}{2}\mathscr{P}^2_{1, 0}.\) Then, the normal form vector field that is normalized up to grade 3 is
  \begin{eqnarray}\nonumber
 &\Theta+2\mathscr{P}_{0, 1}^2\!+\!\mathscr{P}_{0, 1}^1\!+\!2\mathscr{P}_{1, 0}^1\!-\!\mathscr{P}_{1, 0}^2+
 3\mathscr{P}_{ 2, 0}^2+\mathscr{R}_{2, 0}^1
-\frac{1}{3}\mathscr{P}_{2, 1}^1+\frac{2}{3}\mathscr{P}_{2, 1}^2+\mathscr{R}_{3, 0}^1&\\\label{v2.2}
&+\frac{7}{6}\mathscr{P}_{3, 1}^1+\frac{1}{6}\mathscr{P}_{3, 1}^2+\mathscr{P}_{4, 0}^2+\mathscr{R}_{4, 0}^2.&
  \end{eqnarray}

\item [2.] \emph{Third-grade term}s: The second row of the matrix \(\mathcal{A}_{3, 1}^{2, 0}\) is a zero vector, and coefficients follow
\begin{eqnarray*}
      &(\alpha^1_i,\alpha^2_i, \beta^1_i, \beta^2_i)=(0, 0, 0, 0) \hbox{  when  } i=0, 1,&\\
      &(\alpha^1_2,\alpha^2_2, \beta^1_2, \beta^2_2)^\intercal=\left(\widehat{\mathcal{A}_{3, 1}^{2, 0}}\right)^{-1}\left(a^1_{2, 1}, a^2_{3, 0}, b^1_{2, 1}, b^2_{2, 1} \right)^\intercal=
    \Scale[0.8]{\begin{bmatrix}
        4 & -2 & 0 & 0 \\
        2 & 5 & 0 & 0 \\
        0 & 0 & 4 & 0 \\
        0 & 0 & 0 & 4
      \end{bmatrix}^{-1}\begin{bmatrix}
                          -\frac{1}{3} \\
                          1 \\
                          0 \\
                          0
                        \end{bmatrix}=\begin{bmatrix}
                          -\frac{1}{54} \\
                          \frac{7}{54} \\
                          0 \\
                          0
                        \end{bmatrix}  }. &
\end{eqnarray*}

Then, the transformation generator \(X_2:=-\frac{1}{54}\mathscr{P}_{2, 0}^1+\frac{7}{54}\mathscr{P}_{2, 0}^2\) is derived. Therefore, the vector field in \eqref{v2.2} is transformed to
\begin{eqnarray}\nonumber
 &\Theta+2\mathscr{P}_{0, 1}^2\!+\!\mathscr{P}_{0, 1}^1\!+\!2\mathscr{P}_{1, 0}^1\!-\!\mathscr{P}_{1, 0}^2+
 3\mathscr{P}_{ 2, 0}^2+\mathscr{R}_{2, 0}^1+
\frac{2}{3}\mathscr{P}_{2, 1}^2+\frac{2}{27}\mathscr{P}_{3, 0}^1&\\\label{v3.32}
&+\frac{7}{6}\mathscr{P}_{3, 1}^1+\frac{1}{6}\mathscr{P}_{3, 1}^2+\frac{7}{9}\mathscr{P}_{4, 0}^2-\frac{2}{27}\mathscr{R}_{4, 0}^1+\mathscr{R}_{4, 0}^2.&
\end{eqnarray}
  \item[3.]\emph{ Terms of grade \(n\geq 4\)}: Matrix \(\mathcal{A}_{n, 1}^{j, 0}\) is invertible for all integer values \(n\geq 4\) and \(0\leq j\leq n-1\). Hence, the coefficients of the transformation generator for simplifying 4-th grade terms are expressed by
      \begin{eqnarray*}
      &(\alpha^1_i,\alpha^2_i, \beta^1_i, \beta^2_i)=(0, 0, 0, 0) \hbox{  when  } i=0, 1, 2,&\\
      &(\alpha^1_2,\alpha^2_2, \beta^1_2, \beta^2_2)^\intercal=\left({\mathcal{A}_{4, 1}^{3, 0}}\right)^{-1}\left(a^1_{3, 1}, a^2_{3, 0}, b^1_{3, 1}, b^2_{3, 1} \right)^\intercal=
      \Scale[0.8]{\begin{bmatrix}
        6 & -2 & 0 & 0 \\
        0 & 2 & 0 & 0 \\
        0 & 0 & 6 & 0 \\
        0 & 0 & 0 & 6
      \end{bmatrix}^{-1}\begin{bmatrix}
                       \frac{7}{6} \\
                          \frac{1}{6} \\
                          0 \\
                          0
                        \end{bmatrix}=\begin{bmatrix}
                                         \frac{2}{9} \\
                                         \frac{1}{12} \\
                                           0 \\
                                           0
                                       \end{bmatrix}.}
                        &
  \end{eqnarray*}
This gives rise to \(X_3:= \frac{2}{9}\mathscr{P}_{3, 0}^1+\frac{1}{12}\mathscr{P}_{3, 0}^2.\) The latter transforms the vector field \eqref{v3.32} into the truncated (\ie modulo terms of grade five and higher) second level normal form
\begin{eqnarray*}
 &v^{(2)}:=\Theta+2\mathscr{P}_{0, 1}^2\!+\!\mathscr{P}_{0, 1}^1\!+\!2\mathscr{P}_{1, 0}^1\!-\!\mathscr{P}_{1, 0}^2+
 3\mathscr{P}_{ 2, 0}^2+\mathscr{R}_{2, 0}^1+
\frac{2}{3}\mathscr{P}_{2, 1}^2+\frac{2}{27}\mathscr{P}_{3, 0}^1&\\
&-\frac{16}{9}\mathscr{P}_{4, 0}^1-\frac{2}{3}\mathscr{P}_{4, 0}^2-\frac{2}{27}\mathscr{R}_{4, 0}^1+\mathscr{R}_{4, 0}^2.&
\end{eqnarray*}
\end{itemize}
\end{exm}

\subsection{Case III: \(a^1_{0, 1}a^2_{0, 1}a_{1, 0}^1a_{1, 0}^2=0.\)}

When \(a^1_{1, 0}a^2_{1, 0}=0\) and \(a_{0,1}^1a_{0, 1}^2>0,\) the hypotheses in Lemma \ref{lem3} still hold and thus, in those cases, the second level normal forms follows Lemma \ref{lem3}. This is also true for Lemma \ref{lem3-1} when \(a^1_{0, 1}a^2_{0, 1}=0\) and \(a_{1, 0}^1a_{1, 0}^2>0\). When either \(\frac{a_{0, 1}^1}{a_{0, 1}^2}\) or \(\frac{a_{1, 0}^1}{a_{1, 0}^2}\in \mathbb{Q}^c,\) the results follow Lemma \ref{lem38}. Therefore, these cases will not be discussed again in this section.
When \(pa_{0, 1}^1+qa_{0, 1}^2=0\) for two co-prime positive integers \(p\) and \(q\), a zero value for \(a^1_{1, 0}\) or \(a^2_{1, 0}\) reduces the matrix dimension of \(\mathcal{A}_{m(p+q)+i,1}\) for \(i=1,2\) and \(m\in \mathbb{N}\). This influences the second level step of the normal form procedure.
\begin{lem}\label{lem4}
Let \(\frac{a_{0, 1}^1}{a_{0, 1}^2}=\frac{-q}{p}\) for \(\gcd(p, q)=1\) and \(p, q\in \mathbb{N}.\) Consider \(v^{(1)}\) in \eqref{classicalNF}. Then,
\begin{enumerate}
  \item\label{I1} If \(a^2_{1, 0}=0\) and \(a^1_{1, 0}\neq 0,\) the second level normal form is given by
  \begin{eqnarray}\label{SecondLevel5}
& \hspace{-2.2 cm}\Theta+\Sum{k\in\mathbb{N}_2}{}\Sum{i+j= 1}{}\left( b^k_{i,j}\mathscr{R}^k_{i,j}\!+\!a^k_{i,j}\mathscr{P}^k_{i,j}\right)\!+\!a^1_{0, 2}\mathscr{P}_{0, 2}^1\!+\!a^1_{0, 2}\mathscr{P}_{0, 2}^1\!\!+\!a^2_{2,0}\mathscr{P}^2_{2,0}\!+\!b^1_{2,0}\mathscr{R}^1_{2,0}\!+\!b^2_{2,0}\mathscr{R}^2_{2,0}+&\\[-4pt]
&\hspace{-0.5cm} \Sum{m\in \mathbb{N}, k\in \mathbb{N}_2}{}\Scale[0.9]{\Big(\Sum{j=3}{p+q} (a^k_{m(p+q)+j,0}\mathscr{P}^k_{m(p+q)+j, 0}\!+\! b^k_{m(p+q)+j, 0}\mathscr{R}^k_{m(p+q)+j, 0})\!+\!\Sum{k\in\mathbb{N}_2}{}\left(a^k_{mp, mq+k}\mathscr{P}^k_{mp, mq+k}\!+\!b^k_{mp, mq+k}\mathscr{R}^k_{mp, mq+k}\right)\Big)}.&\nonumber
\end{eqnarray}
  \item \label{I2} If \(a^1_{1, 0}=0\) and \(a^2_{1, 0}\neq 0,\) the second level normal form of \(v^{(1)}\) is
  \begin{eqnarray}\label{SecondLevel6}
&\hspace{-2 cm} \Theta\!+\!\Sum{k\in\mathbb{N}_2}{}\left( b^k_{0, 1}\mathscr{R}^k_{0, 1}\!+\!a^k_{0, 1}\mathscr{P}^k_{0, 1}\right)\!+\!\Sum{m\in\mathbb{Z}_{\geq 0}}{}\Scale[0.95]{\Sum{k\in\mathbb{N}_2, j\in\mathbb{N}_{p+q}}{} (a^k_{m(p+q)+j,0}\mathscr{P}^k_{m(p+q)+j, 0}\!+\! b^k_{m(p+q)+j, 0}\mathscr{R}^k_{m(p+q)+j, 0})}&\\[-6pt]
&\nonumber+\Sum{m\in\mathbb{Z}_{\geq 0}}{}\Sum{j\in Q, k\in\mathbb{N}_2}{}(a^k_{j, m(p+q)+1-j}\mathscr{P}^k_{j, m(p+q)+1\!-\!j}\!+\!b^k_{j, m(p+q)+1\!-\!j}\mathscr{R}^k_{j, m(p+q)+1-j})\!+\!\Sum{m\in\mathbb{Z}_{\geq 0}}{}a^2_{mp, mq+2}\mathscr{P}^2_{mp, mq+2}, &
\end{eqnarray} where \(Q:=\{mp, m(p+q)+1\}\) and \(a^1_{mp, mq+1}=a^1_{m(p+q)+1, 0}=0.\)
  \end{enumerate}
\end{lem}
\begin{proof}
Item 1. Since \(a^2_{1, 0}=0\), matrices of \(\mathcal{A}_{n, 1}^{0, 1}\) and \(\mathcal{A}_{n, 1}^{1, 1}\) for every integer \(n\geq 2\) are singular while the matrices \(\mathcal{A}_{n, 1}^{i, 1}\) for \(2\leq i\leq n-1\) are invertible. Recall that \(n\) is the grade of the intended terms for normalization. If \(n=m(p+q)+k\) for \(m\in \mathbb{N}, k=1, 2,\) the matrices \(\mathcal{A}_{n, 1}^{mp, 0}\) and \(\mathcal{A}_{n+1, 1}^{mp, 0}\) are singular due to a similar argument in the proof of Lemma \ref{Lem2.4}.
Further, matrices \(\mathcal{A}_{n+k, 1}^{i, 0}\) for \(0\leq i\leq n+k-1, i\neq mp\) and \(k=1, 2\) are invertible. These are associated with
\(n\in P_5\) and \(n\in P_6\) for \(k=1,2\) in Lemma \ref{Lem2.4}. When \(3\leq k\leq p+q,\) matrices \(\mathcal{A}_{n, 1}^{i, 0}\) for \(0\leq i\leq n-1\) are invertible and this leads to \(n\in P_7.\) Therefore, the second level normal form follows equation \eqref{SecondLevel5}.

Item 2. Let \(n=m(p+q)+1\). Hence, the diagonal entries of \(\mathcal{A}_{n, 1}^{n-1, 1}\) for every integer \(n\geq 2\) are zero. Yet, the matrices \(\mathcal{A}_{n, 1}^{k, 0}\) for \(0\leq k\leq n-1, k\neq mp\) are invertible. Due to the degeneracy of matrices \(\mathcal{A}_{n, 1}^{mp, 0}\) and \(\mathcal{A}_{n, 1}^{n-1, 1},\) \(\Null(\mathcal{A}_{n, 1})=2.\)
Coefficients for \(0\leq i\leq mp-1\) and \( mp+1\leq i\leq n-2\) are derived from \eqref{AscRec},
\begin{eqnarray*}\label{eq8}
&\alpha_{n-1}^1=\frac{-a^2_{n, 0}}{2a^2_{1, 0}}, \quad \alpha_{n-1}^2=\frac{a^1_{n-1, 1}-\alpha_{n-1}^1(n-1)a^1_{0, 1}-\alpha^1_{n-2}a^2_{1, 0}}{2a^1_{0, 1}},\quad \beta^1_{n-1}=\beta^2_{n-1}=0,&
\end{eqnarray*}
and values of \((\alpha^1_{mp}, \alpha^2_{mp}, \beta^1_{mp}, \beta^2_{mp})\) are obtained from equations \eqref{eq5}. However, \(mr\) must be replaced with \(n-1.\)
Therefore, Eulerian and rotational terms  \(\mathscr{P}^2_{mp, mq+1}, \mathscr{P}^1_{n, 0}, \mathscr{R}^1_{j, n-j}, \mathscr{R}^2_{j, n-j}\) when \(j=mp, n\) cannot  necessarily be eliminated from \(v^{(1)}\) in the second level step. Assume that \(n=m(p+q)+2.\)  Due to the singularities of \(\mathcal{A}_{n, 1}^{mp, 0}\) and \(\mathcal{A}_{n, 1}^{n-1, 1},\) \(\rank (\mathcal{A}_{n, 1})= 4n-1.\) Since \(\mathcal{A}_{n, 1}^{i, 0}\) is invertible for \(0\leq i\leq n-1\) and \(i\neq mp\), the relations \eqref{AscRec} and \eqref{eq7} when \(0\leq i\leq n-1\) and \(i\neq mp\) lead to the elimination of all terms \(\mathscr{P}^k_{j, n-j}, \mathscr{R}^k_{j, n-j}\) for \(k=1, 2, 0\leq j\leq n-1\) with only one exception, that is, \(\mathscr{P}^2_{mp, mq+1}.\) Hence, the second level normal form is expressed as in equation \eqref{SecondLevel6}.
\end{proof}

\begin{lem}Let \(a^1_{0, 1}=0\) and \(a^2_{0, 1}\neq 0.\)
\begin{enumerate}
\item When \(a^1_{1, 0}=0\) and \(a^2_{1, 0}\neq 0,\) the second level normal form is expressed by
\begin{eqnarray}\label{SecondLevel7-1}
\Theta\!+\!\Sum{k\in\mathbb{N}_2}{}\Sum{i+j= 1}{}\left(a^k_{i,j}\mathscr{P}^k_{i,j}\!+\!b^k_{i,j}\mathscr{R}^k_{i,j}\right)+\!\Sum{k\in\mathbb{N}_2}{}\Sum{n=2}{\infty}\Sum{j=0}{2}\Scale[0.95]{ (a^k_{n-j,j}\mathscr{P}^k_{n-j,j}\!+\! b^k_{n-j,j}\mathscr{R}^k_{n-j,j})},
\end{eqnarray}
 where \(a_{n-2, 2}^1=b_{n-2, 2}^1=b_{n-2, 2}^2=a_{n-1, 1}^2=a_{n, 0}^2=0\) for every integer number \(n\geq 2.\)
\item Assume that \(a^2_{1, 0}=0\) and \(a^1_{1, 0}\neq 0.\) Therefore, the second level normal form follows
\begin{eqnarray}\label{SecondLevel7-2}
&\hspace{-1 cm}\Theta\!+\!\Sum{k\in\mathbb{N}_2}{}\Sum{i+j= 1}{}\left(a^k_{i,j}\mathscr{P}^k_{i,j}\!+\!b^k_{i,j}\mathscr{R}^k_{i,j}\right)\!+\!a^2_{0, 2}\mathscr{P}^2_{0, 2}\!+\!a^2_{2, 0}\mathscr{P}^1_{2, 0}\!+\!\Sum{k\in\mathbb{N}_2}{}(a^k_{1, 1}\mathscr{P}^k_{1, 1}\!+\!b^k_{1, 1}\mathscr{R}^k_{1, 1})\!+\!a^1_{2, 1}\mathscr{P}^1_{2, 1}&\\
&\nonumber +\!\Sum{k\in\mathbb{N}_2}{}\Sum{n=3}{\infty} (a^k_{n-2,2}\mathscr{P}^k_{n-2,2}\!+\! b^k_{n-2,2}\mathscr{R}^k_{n-2,2}),\quad\hbox{where}\quad a^1_{1, 2}=0.&
\end{eqnarray}
\end{enumerate}
\end{lem}
\begin{proof}
Item 1. Since \(a^1_{0, 1}=0\) and \(a^2_{0, 1}\neq 0,\) the second row of matrix \(\mathcal{A}_{n ,1}^{n-2, 0},\) the first, third and fourth columns of \(\mathcal{A}_{n ,1}^{n-1, 0}\) are all zero vectors. From \(a^1_{1, 0}=0\) and \(a^2_{1, 0}\neq 0,\) we deduce that all diagonal entries of \(\mathcal{A}_{n ,1}^{n-1, 1}\) are zero.
Thereby, \(\ker(\mathcal{A}_{n ,1})=4n-3.\) When \(n\geq 3,\) due to the non-degeneracy of matrices \(\mathcal{A}_{n ,1}^{i, 0}\) for \(0\leq i\leq n-3,\) the coefficients can be assigned through \eqref{AscRec}. The remaining coefficients are determined by \(\alpha^2_{n-2}=\beta^1_{n-1}=\beta^2_{n-1}:=0\),
\begin{eqnarray*}
&\alpha^1_{n-2}=\frac{a^1_{n-2,2}-4a^2_{1, 0}\alpha^1_{n-3}}{2a^2_{0, 1}}, \quad
      (\beta^1_{n-2}, \beta^2_{n-2})=\frac{(b^1_{n-2,1}, b^2_{n-2,1})-4a^2_{1, 0}(\beta^1_{n-3}, \beta^2_{n-3})+2(b^1_{1, 0}, b^2_{1, 0})\alpha^1_{n-3}}{2a^2_{0, 1}},&\\
&\alpha^1_{n-1}=-\frac{a^2_{n,0}}{2a^2_{0, 1}},\quad \alpha^2_{n-1}=-\frac{a^2_{n-1,1}-2a^2_{1, 0}\alpha^2_{n-2}+2b^2_{1, 0}\alpha^1_{n-2}}{2a^2_{0, 1}}.&
\end{eqnarray*}
For \(n=2,\) the coefficients are derived similar the above relations where \(\alpha^1_{n-3}=0.\)
 Hence, terms \(\mathscr{P}^2_{n-2, 2},\mathscr{P}^1_{n-j, j}, \mathscr{R}^k_{n-j, j}\) for \(j=0, 1\) and \(k=0, 1,\) cannot  necessarily be simplified.

Item 2. The second, third and fourth columns of \(\mathcal{A}_{n ,1}^{0, 1}\) are zero vectors. Further, the \((1, 1)\)-entry from \(\mathcal{A}_{n ,1}^{1, 1}\) is zero. Let \(n=2.\) Except the second and fifth columns of \(\mathcal{A}_{n, 1},\) the other columns are linearly independent. Thus, six terms of grade 2 can be simplified. To this end, it suffices to assign the coefficients as \(\alpha^2_{0}=\alpha^1_{1}:=0,\)
\begin{eqnarray*}
&\left(\alpha^1_0, \beta^1_0, \beta^2_0\right)=\frac{\left(a^1_{0, 2}, b^1_{0, 2}, b^1_{0, 2}\right)}{2a_{0, 1}^2},\quad\hbox{ and }\quad \left(\alpha^2_1, \beta^1_1, \beta^2_1\right)=\frac{\left(a^2_{2, 0}, b^1_{2, 0}, b^1_{2, 0}\right)}{2a_{1, 0}^2}.&
\end{eqnarray*} Hence, terms \(\mathscr{P}^1_{0, 2},\) \(\mathscr{P}^2_{2, 0},\) \(\mathscr{R}^k_{2, 0}\) and \(\mathscr{R}^k_{0, 2}\) for \(k=1, 2\) are simplified. Assume that \(n\geq 3.\) Here,  \(\rank(\mathcal{A}_{n, 1})=4n\) due to \(\ker(\mathcal{A}_{n, 1})=\{0\}.\) When \(n=3,\) coefficients are computed by the relations
\begin{eqnarray*}
&(\alpha^1_{0}, \alpha^2_{0}, \beta^1_{0}, \beta^2_{0})^\intercal\!=\!(\mathcal{A}_{n ,1}^{0, 0})^{-1}(a^1_{0, 3}, a^2_{0, 3}, b^1_{0, 3}, b^2_{0, 3})^\intercal,\,\,
(\alpha^1_{2}, \alpha^2_{2}, \beta^1_{2}, \beta^2_{2})^\intercal\!=\!(\mathcal{A}_{n ,1}^{2, 1})^{-1}(a^1_{3, 0}, a^2_{3, 0}, b^1_{3, 0}, b^2_{3, 0})^\intercal,&\\
&\alpha^1_1=\frac{a^1_{1, 2}+2a^1_{1, 0}\alpha^1_0}{2a^2_{0, 1}}, \quad (\alpha^2_{1}, \beta^1_1, \beta^2_1)=\frac{(a^2_{2, 1}, b^1_{2, 1}, b^2_{2, 1})+2(a_{0, 1}^2, b^1_{0, 1}, b^2_{0, 1})}{2a^2_{1, 0}}.&
\end{eqnarray*} Thereby, terms \(\mathscr{P}^2_{1, 2}, \mathscr{P}^1_{2, 1}, \mathscr{R}^k_{1, 2}\) for \(k=1, 2\) cannot be simplified. Let \(n>3.\)  For \(0\leq i\leq n-3,\) we use the relation \eqref{AscRec} while when \(n\geq i\geq n-1\) the relation \eqref{DecRec} is applied. Therefore, terms \(\mathscr{P}^k_{n-2, 2}\) and \(\mathscr{R}^k_{n-2, 2}\) for \(k=1, 2\) may remain in the normal form system.
\end{proof}

\begin{lem} Assume that \(a^2_{0, 1}=0\) and \(a^1_{0, 1}\neq 0.\)
\begin{enumerate}
\item If \(a^1_{1, 0}=0\) and \(a^2_{1, 0}\neq 0,\) the second level normal form is given by
\begin{eqnarray}\label{SecondLevel8}
\hspace{-0.65 cm}\Scale[0.95]{\Theta\!+\!a^1_{0, 1}\mathscr{P}^1_{0, 1}\!+\!a^2_{1, 0}\mathscr{P}^2_{1, 0}\!+\!\Sum{k\in\mathbb{N}_2}{}b^k_{k-1, 2-k}\mathscr{R}^k_{k-1, 2-k}\!+\!\Sum{n=2}{\infty} (a^2_{0, n}\mathscr{P}^2_{0, ,0}\!+\!a^1_{n, 0}\mathscr{P}^1_{n, ,0}\!+\! \Sum{j\in\{0,n\}}{}\Sum{k\in\mathbb{N}_2}{}b^k_{n-j,j}\mathscr{R}^k_{n-j,j})}.
\end{eqnarray}
\item When \(a^2_{1, 0}=0\) and \(a^1_{1, 0}\neq 0,\) the second level normal form follows
\begin{eqnarray}\label{SecondLevel9}
\hspace{-0.65 cm}\Scale[0.88]{\Theta\!+\!a^1_{0, 1}\mathscr{P}^1_{0, 1}\!+\!a^1_{1, 0}\mathscr{P}^1_{1, 0}\!+\!\!\Sum{k\in\mathbb{N}_2}{}b^k_{k-1, 2-k}\mathscr{R}^k_{k-1, 2-k}+\!\!\Sum{n=2}{\infty} \big(a^2_{0, n}\mathscr{P}^2_{0, n}\!+\!\!\Sum{k\in\mathbb{N}_2}{}(b^k_{2, n-2}\mathscr{R}^k_{2, n-2}\!+\!a^k_{2, n-2}\mathscr{P}^k_{2, n-2}\!+\! b^k_{0, n}\mathscr{R}^k_{0, n})\big)}.
\end{eqnarray}
\end{enumerate}
\end{lem}
\begin{proof}
Claim 1. By \(a^2_{0, 1}=0\) and \(a^1_{0, 1}\neq 0,\) the first, third and fourth columns of matrix \(\mathcal{A}_{n ,1}^{0, 0}\) are zero vectors. Further, \(a^1_{1, 0}=0\) and \(a^2_{1, 0}\neq 0\) imply that all diagonal entries of \(\mathcal{A}_{n ,1}^{n-1, 1}\) are zero. From \(\Null(\mathcal{A}_{n, 1})=2,\) we conclude that \(\rank(\mathcal{A}_{n, 1})=4n-2\) for every \(n\ge 2.\) Then, coefficients are derived as
\begin{eqnarray*}
&\alpha^2_0=\frac{a^1_{0, n}}{2a^1_{0, 1}}, \quad \alpha^1_{n-1}=\frac{a^2_{n, 0}}{2a^2_{1, 0}}, \quad \beta^1_{n-1}=\beta^2_{n-1}=0.&
\end{eqnarray*} The coefficients \(\alpha^1_0\) and \(\alpha^2_{n-1}\) are derived from the first two rows of equation \eqref{eq5} where the values of \(mp\) and \(mr\) are respectively replaced by \(0\) and \(n-1\). Other coefficients for \(1\leq i\leq n-2\) and \(n\geq 3\) are determined by the relation \eqref{AscRec}. Hence, terms \(\mathscr{P}^2_{0, n}, \mathscr{P}^1_{n, 0}\) and \(\mathscr{R}^k_{n-j, j}\) for \(k=1, 2, j=0, n\) may remain in the second level normal form.

Claim 2. For every \(n\geq 2,\) we have \(\Null(\mathcal{A}_{n, 1})=4n-3.\) We take \(\alpha^1_0=\beta^1_0=\beta^2_0:=0\) to obtain
\begin{eqnarray*}
& \alpha^1_{1}=\frac{a^1_{0, n}}{2a^1_{0, 1}},\quad\hbox{ and }\quad (\alpha^1_{1}, \alpha^2_{1}, \beta^1_{1}, \beta^2_{1})^\intercal\!=\!(\mathcal{A}_{n ,1}^{1, 0})^{-1}(a^1_{1, n-1}, a^2_{1, n-1}, b^1_{1, n-1}, b^2_{1, n-1})^\intercal.&
\end{eqnarray*} Hence, the vector fields \(\mathscr{P}^1_{0, n}, \mathscr{P}^k_{j, n-j}\) and \(\mathscr{R}^k_{j, n-j}\) for \(k=1, 2, 0\leq j\leq n, j\neq 0, 2\) are simplified using the recursive relation \eqref{DecRec} when \(3\leq i\leq n.\)
\end{proof}
\begin{lem} Let \(a^1_{0, 1}=a^2_{0, 1}= 0\) and consider the differential system \eqref{pre-classical}.
\begin{enumerate}
  \item[I.] If \(a^1_{1, 0}=0\) and \(a^2_{1, 0}\neq 0,\) the second level normal form is derived as
  \begin{eqnarray*}\label{SecondLevel10}
  \hspace{-0.4 cm}\Theta+ a^2_{1, 0}\mathscr{P}^2_{1, 0}+\Sum{k\in\mathbb{N}_2}{}\Sum{j\in\{0,1\}}{}b^k_{j,1-j}\mathscr{R}^k_{j,1-j}+\!\Sum{k\in\mathbb{N}_2}{}\Sum{j\in\{0,n\}}{}\Sum{n=2}{\infty} (a^k_{j, n-j}\mathscr{P}^k_{j, n-j}\!+\! b^k_{n-j,j}\mathscr{R}^k_{n-j,j}), \hbox{ where }a^2_{n, 0}=0.
  \end{eqnarray*}
  \item[II.] When \(a^2_{1, 0}=0\) and \(a^1_{1, 0}\neq 0,\) the second level normal form follows
  \begin{eqnarray*}
  \hspace{-0.4 cm}\Theta+ a^1_{1, 0}\mathscr{P}^1_{1, 0}+\Sum{k\in\mathbb{N}_2}{}\Sum{j\in\{0,1\}}{}b^k_{j,1-j}\mathscr{R}^k_{j,1-j}+\!\Sum{k\in\mathbb{N}_2}{}\Sum{j\in\{0,1\}}{}\Sum{n=2}{\infty} (a^k_{j, n-j}\mathscr{P}^k_{j, n-j}\!+\! b^k_{n-j,j}\mathscr{R}^k_{n-j,j})+\!\Sum{n=2}{\infty}a^1_{2, n-2}\mathscr{P}^1_{2, n-2},
  \end{eqnarray*} where \(a^1_{1, n-1}=0.\) For \(b^1_{0, 1}\neq 0,\) \(b^1_{0, n}=0,\) \(b^1_{0, 1}=0,\) and if \(b^2_{0, 1}\neq 0,\) we have \(b^2_{0, n}=0.\)
\end{enumerate}
\end{lem}
\begin{proof} Claim I. From the hypothesis \(a^1_{0, 1}=a^2_{0, 1}= 0,\)  except entries \((3, 2)\) and \((4, 2)\) from \(\mathcal{A}_{n, 1}^{i, 0}\), other elements are zero for \(0\leq i\leq n-1.\) Due to \(a^1_{1, 0}=0\) and \(a^2_{1, 0}\neq 0,\) the diagonal entries of \(\mathcal{A}_{n, 1}^{n-1, 1}\) are zero. Therefore, \(\rank(\mathcal{A}_{n, 1})=4n-3.\) Thus, we take \(\alpha^1_{n-1}=\frac{a^2_{n, 0}}{2a^2_{1, 0}},\) \(\alpha^2_{n-1}=\beta^1_{n-1}=\beta^2_{n-1}=0,\) and other coefficients are obtained by the relation  \eqref{DecRec} for \(1\leq i\leq n-1.\) Hence, terms \(\mathscr{P}^k_{0, n}, \mathscr{R}^k_{0, n}, \mathscr{R}^k_{n, 0}\) for \(k=1, 2\) and \(\mathscr{P}^1_{n, 0}\) cannot  necessarily be removed from the first level normal form.

Claim II. The conditions \(a^2_{1, 0}=0\) and \(a^1_{1, 0}\neq 0\) imply that the second, third and fourth columns of \(\mathcal{A}_{n, 1}^{0, 1}\) and entry \((1, 1)\) of matrix \(\mathcal{A}_{n, 1}^{1, 1}\) are all zero. When \(|b^1_{0, 1}|+|b^2_{0, 1}|\neq 0,\) \(\Null(\mathcal{A}_{n, 1})=3.\) Thereby, \(4n-3\)-number of terms with grade \(n\) can be simplified from the first level normal forms when \(|b^1_{0, 1}|+|b^2_{0, 1}|\neq 0.\) If the latter condition fails, \ie \(b^1_{0, 1}=b^2_{0, 1}= 0,\) we have \(\Null(\mathcal{A}_{n, 1})=3.\) This implies that  \(\rank(\mathcal{A}_{n, 1})=4n-4.\) Here, coefficients follow \eqref{DecRec} where \(3\leq i\leq n.\) Furthermore, let \(\alpha^1_1=\beta^1_0=\beta^2_0=0,\)
\begin{eqnarray*}
&\alpha^1_{0}=-\frac{a^1_{1, n-1}}{2a_{1, 0}^1},\quad\alpha^2_{1}=\frac{a^2_{2, n-2}}{2a_{1, 0}^1},\quad \left(\beta^1_{1}, \beta^2_1\right)=\frac{\left(b^1_{2, n-2}, b^2_{2, n-2}\right)+2\alpha^2_{2}\left(b^1_{0, 1},b^2_{0, 1}\right)}{2a_{1, 0}^1}. &
\end{eqnarray*}
For \(b^1_{0, 1}\neq 0,\) we take \(\alpha^2_{0}=-\frac{b^1_{0, n}}{2b_{0, 1}^1}.\) If \(b^1_{0, 1}=0\) and \(b^2_{0, 1}\neq 0,\) let \(\alpha^2_{0}=-\frac{b^2_{0, n}}{2b_{0, 1}^2}.\) Finally, we take \(\alpha^2_{0}=0\) for \(b^1_{0, 1}=b^2_{0, 1}=0.\)
Hence, \(\mathscr{P}^1_{1, n-1},\) \(\mathscr{P}^k_{j, n-j}\) and \(\mathscr{R}^k_{j, n-j}\) for \(k=1, 2, 2\leq j\leq n\) (except \(\mathscr{P}^1_{2, n-2}\)) are simplified. Further, \(\mathscr{R}^1_{0, n}\)  is removed from the normal form system if \(b^1_{0, 1}\neq 0\) while we can normalize \(\mathscr{R}^2_{0, n}\) when \(b^1_{0, 1}=0\) and \(b^2_{0, 1}\neq 0\).
\end{proof}
\begin{rem} The remaining cases (\(a_{1, 0}^1a_{1, 0}^2\neq 0,\) \(a_{0, 1}^1a_{0, 1}^2=0\)) and (\(a_{1, 0}^1=a_{1, 0}^2= 0,\) \(|{a_{0, 1}^1}| + |{a_{0, 1}^2}|\neq 0\)) are treated similar to the preceding cases. The argument applies the properties of \(\mathcal{A}_{n, 1}^{j, k}\) for \(k=0, 1, 0\leq j\leq n-1\) and the discussion on the rank of \(\mathcal{A}_{n, 1}.\) We, however, skip them here for briefness; see Example \ref{Exam_CaseIII}.
\end{rem}
\begin{exm}\label{Exam_CaseIII}
Consider a first level normal form given by \eqref{Example2} where \(a_{0, 1}^1=a_{0, 1}^2=0\) and \(a_{1, 0}^1=2, a_{1, 0}^2=1\). Hence, matrix \(\mathcal{A}_{n, 1}^{j, 0}\) is a zero matrix for all \(j\in \mathbb{N}\) while \(\mathcal{A}_{n, 1}^{j, 1}\) is invertible for \(0\leq j\leq n-1\) and \(n\geq 3\) (except \(\mathcal{A}_{3, 1}^{0, 1}\)). Normalization of terms up to the third grade is as follows:
\begin{itemize}
  \item[i.] \emph{Homogeneous terms of grade two}: Coefficients are derived via
  \begin{eqnarray*}
 &(\alpha^1_{0}, \alpha^2_0, \beta^1_0, \beta^2_0)^\intercal=\left(\mathcal{A}_{2, 1}^{0, 1}\right)^{-1}(1, 1, 0, 0)^\intercal=\left(-\frac{1}{2}, 0, 0, 0\right)^\intercal, \quad (\alpha^1_{1}, \alpha^2_1, \beta^1_1, \beta^2_1)=\left(0, \frac{1}{4}, 0, 0\right).&
  \end{eqnarray*}
  Hence, the transformation generator is \(X_2:=-\frac{1}{2}\mathscr{P}_{0, 1}^1-\frac{1}{4}\mathscr{P}_{1, 0}^2\) and \(v^{(1)}\) is transformed to
  \begin{eqnarray*}
 &v^{(2, 2)}:=\Theta\!+\!\mathscr{P}_{1, 0}^1\!+\!2\mathscr{P}_{1, 0}^2
-\mathscr{P}_{1, 2}^2+\frac{5}{3}\mathscr{P}_{2, 1}^1-\frac{4}{3}\mathscr{P}_{2, 1}^2-\frac{5}{2}\mathscr{R}_{2, 1}^1+\mathscr{P}_{3, 0}^2.&
  \end{eqnarray*}
  \item[ii.] \emph{Third-grade terms}: The first row of the matrix \(\mathcal{A}_{3, 1}^{0, 0}\) is a zero vector, and all other matrices are invertible. Then, coefficients are assigned as \((\alpha^1_{0}, \alpha^2_0, \beta^1_0, \beta^2_0)^\intercal=(0, -\frac{1}{4}, 0, 0),\)
  \begin{eqnarray*}
  &(\alpha^1_{1}, \alpha^2_1, \beta^1_1, \beta^2_1)^\intercal=\left(\mathcal{A}_{3, 1}^{1, 1}\right)^{-1}(\frac{5}{3}, -\frac{4}{3}, -\frac{5}{2}, 0)^\intercal=\left(\frac{5}{6}, \frac{1}{18}, -\frac{5}{12}, 0\right)^\intercal,&\\
   &\hbox{ and }\quad(\alpha^1_{2}, \alpha^2_2, \beta^1_2, \beta^2_2)^\intercal=\left(\mathcal{A}_{3, 1}^{2, 1}\right)^{-1}(0, 1, 0, 0)^\intercal=\left(0, \frac{1}{8}, 0, 0\right)^\intercal.&
  \end{eqnarray*} Thereby, the transformation generator is derived as \(X_2:=-\frac{1}{4}\mathscr{P}_{0, 2}^2+\frac{5}{6}\mathscr{P}_{1, 1}^1+\frac{1}{18}\mathscr{P}_{1, 1}^2-\frac{5}{12}\mathscr{R}_{1, 1}^1+\frac{1}{8}\mathscr{P}_{2, 0}^2.\) Therefore, the second level normal form of \(v^{(1)}\) is given by
  \begin{eqnarray}\label{Simpl319}
 &v^{(2)}:=\Theta\!+\!\mathscr{P}_{1, 0}^1\!+\!2\mathscr{P}_{1, 0}^2+ {\sc h.o.t.},&
  \end{eqnarray} where higher order terms stand for terms of degree nine (grade 4) and higher.
\end{itemize}
\end{exm}

\section{The infinite level normal forms}\label{sec4}

In this section, we deal with the infinite level normal forms of the differential system \eqref{pre-classical}.
\begin{thm}\cite[Therorem 4.1]{GazorHamziShoghi}\label{thm1}
Assume that hypotheses of Lemma \ref{lem3} are satisfied. Furthermore, let \(a_{1, 0}^1=a_{1, 0}^2=b^1_{1, 0}=b^2_{1, 0}=0, \frac{a_{0, 1}^2}{a_{0, 1}^1}\notin \mathbb{N}.\) Then, there are an infinite sequence of near identity polynomial changes of state variables to transform the differential system \eqref{pre-classical} into its infinite level normal form given by
\begin{eqnarray}\label{inftyLevel1}
\nonumber&\frac{d}{dt}{z_1}=I\omega_1z_1+(a_{0, 1}^1+Ib_{0, 1}^1)z_1|z_2|^{2}+a_{0, 2}^1z_1|z_1|^{4}+\Sum{j\geq 3}{}(a_{0, j}^1+Ib_{0, j}^1)z_1|z_2|^{2j},&\\
&\frac{d}{dt}{z_2}=I\omega_2z_2+(a_{0, 1}^2+Ib_{0, 1}^2)z_2|z_2|^{2}+\Sum{j\geq 2}{}(a_{0, j}^2+Ib_{0, j}^2)z_2|z_2|^{2j},&
\end{eqnarray} and \(\frac{d}{dt}(w_1, w_2)=\frac{d}{dt}(\overline{z_1}, \overline{z_2}).\) Furthermore, the normal form coefficients \(a_{i, j}\) and \(b_{i, j}\) in \eqref{inftyLevel1} are uniquely determined by polynomial functions in terms of coefficients in \eqref{pre-classical}.
\end{thm}

\begin{thm}\cite[Theorem 4.2]{GazorHamziShoghi}\label{thm2} Assume that the hypothesis of Lemma \ref{lem3} hold and consider the differential system
\begin{eqnarray}\label{inftyLevel2}
&\nonumber\frac{d}{dt}\rho_1=\rho_1\Big( a^1_{1, 0}{\rho_1}^{2}+a^1_{0, 1}{\rho_2}^{2}+a^1_{2,0}{\rho_1}^{4}+\Sum{j=2}{\infty}a^1_{j,0}{\rho_1}^{2j}\Big),\, \frac{d}{dt}\rho_2=\rho_2\Big( a^2_{1, 0}{\rho_1}^{2}+a^2_{0, 1}{\rho_2}^{2}+\Sum{j=2}{\infty}a^2_{j,0}{\rho_1}^{2j}\Big), &\\
&\frac{d \vartheta_k}{dt}=\omega_k+b^k_{1, 0}{\rho_1}^{2}+b^k_{0, 1}{\rho_2}^{2}+ \Sum{j=2}{\infty} b^\ell_{j,0}\rho_1^{2j},\quad \text{for}\quad k=1, 2.&
\end{eqnarray}
System \eqref{inftyLevel2} constitutes an infinite level normal form for the differential system \eqref{pre-classical} such that
\begin{enumerate}
  \item \(a_{j, 0}^1= 0\) if \(a_{1, 0}^1a_{j-1, 0}^1\neq 0\) for every \(j\geq 3.\)
  \item  \(a_{j, 0}^2= 0\) when \(a_{j-m, 0}^1=0, a_{m, 0}^1a_{j-m, 0}^2\neq 0\)  for \(m=1, j-1\) and every \(j\geq 3\).
  \item If \(a_{j-m, 0}^1=a_{j-m, 0}^2=0, a_{m, 0}^1b_{j-m, 0}^1\neq 0\) for \(m=1, j-1.\) Then, \(b_{j, 0}^1= 0\) for every \(j\geq 3\).
  \item  Whenever \(a_{j-m, 0}^1=b_{j-m, 0}^1=0, a_{m, 0}^1b_{j-m, 0}^2\neq 0\)   for \(m=1, j-1,\) \(b_{j, 0}^1= 0\) holds for every \(j\geq 3.\)
\end{enumerate}
\end{thm}
When the hypotheses in Theorems \ref{thm1} and \ref{thm2} fail, the obtained normal forms are not necessarily infinite level normal forms. The following counterexample illustrates this.

\begin{exm}[A third level normal form] Consider a second level normal form vector field given by
\begin{eqnarray*}
v^{(2)}:=\Theta+ \mathscr{P}^1_{0, 1}+\mathscr{P}^1_{1, 0}-4\mathscr{P}^2_{0, 1}-4\mathscr{P}^2_{1, 0}+\mathscr{P}^1_{2, 0}+\mathscr{P}^1_{3, 0}+2\mathscr{P}^2_{3, 0}+\mathscr{R}^1_{3, 0}+\mathscr{R}^2_{3, 0}+h.o.t..
\end{eqnarray*} Then,
\begin{eqnarray*}
\ker(d^{2, 2})=\left\{u\in \mathcal{G}_{2}| \left[u,\mathscr{P}^1_{0, 1}\!+\!\mathscr{P}^1_{1, 0}\!-\!4\mathscr{P}^2_{0, 1}\!-\!4\mathscr{P}^2_{1, 0}\right]\!=\!0\right\}=\mathbb{R}\left\{\mathscr{P}^1_{0, 1}\!+\!\mathscr{P}^1_{1, 0}\!-\!4\mathscr{P}^2_{0, 1}\!-\!4\mathscr{P}^2_{1, 0}\right\}.
\end{eqnarray*} Let \(X:=\gamma(\mathscr{P}^1_{0, 1}+\mathscr{P}^1_{1, 0}-4\mathscr{P}^2_{0, 1}-4\mathscr{P}^2_{1, 0})+\sum_{j=0}^{2}\sum_{k=1}^{2}\left(\alpha_{j}^k\mathscr{P}^k_{j, 2-j}+\beta^k_j\mathscr{R}^k_{j, 2-j}\right)\) be the transformation generator intended for normalizing grade-three terms via the third level normal form procedure. Thus, \(X\) transforms the vector field \(v^{(2)}\) to
\begin{eqnarray*}
& v^{(3)}:= \Theta+ \mathscr{P}^1_{0, 1}+\mathscr{P}^1_{1, 0}-4\mathscr{P}^2_{0, 1}-4\mathscr{P}^2_{1, 0}+\mathscr{P}^1_{2, 0}+\mathscr{P}^1_{3, 0}+2\mathscr{P}^2_{3, 0}+\mathscr{R}^1_{3, 0}+\mathscr{R}^2_{3, 0}+&\\[-4pt]
&\quad\quad d^{3, 3}\Big(\gamma(\mathscr{P}^1_{0, 1}+\mathscr{P}^1_{1, 0}-4\mathscr{P}^2_{0, 1}-4\mathscr{P}^2_{1, 0}), \Sum{j=0}{2}\Sum{k=1}{2}\big(\alpha_{j}^k\mathscr{P}^k_{j, 2-j}+\beta^k_j\mathscr{R}^k_{j, 2-j}\big)\Big)+h.o.t..&
\end{eqnarray*}
Furthermore,
\begin{eqnarray*}
& d^{3, 3}\Big(\gamma(\mathscr{P}^1_{0, 1}+\mathscr{P}^1_{1, 0}-4\mathscr{P}^2_{0, 1}-4\mathscr{P}^2_{1, 0}),\Sum{j=0}{2}\Sum{k=1}{2}\left(\alpha_{j}^k\mathscr{P}^k_{j, 2-j}+\beta^k_j\mathscr{R}^k_{j, 2-j}\right)\Big)=& \\ [-4pt]
&\Scale[0.95]{\hspace{-0.5 cm} \gamma\Big[\mathscr{P}^1_{0, 1}\!+\!\mathscr{P}^1_{1, 0}\!-\!4\mathscr{P}^2_{0, 1}\!-\!4\mathscr{P}^2_{1, 0}, \mathscr{P}^1_{2, 0}\Big]\!+\!\Sum{j=0}{2}\Sum{k=1}{2}\Big[ \alpha_{j}^k\mathscr{P}^k_{j, 2-j}+\beta^k_j\mathscr{R}^k_{j, 2-j},  \mathscr{P}^1_{0, 1}\!+\!\mathscr{P}^1_{1, 0}\!-\!4\mathscr{P}^2_{0, 1}\!-\!4\mathscr{P}^2_{1, 0}\Big].}&
\end{eqnarray*} We have \(\rank(d^{3, 3})=10\) and \(\rank(d^{3, 2})=9\). We choose \(\alpha^k_j=\beta^k_j=0\) for \(k=1, 2, 0\leq j\leq 1\) and \((\alpha^1_2, \alpha^2_2,\beta^1_2, \beta^2_{2}, \gamma)=(\frac{-1}{4},0, 0, 0,\frac{3}{8}).\) Thereby, the third level normal form is written as
   \begin{eqnarray*}
v^{(3)}=\Theta+ \mathscr{P}^1_{0, 1}+\mathscr{P}^1_{1, 0}-4\mathscr{P}^2_{0, 1}-4\mathscr{P}^2_{1, 0}+\mathscr{P}^1_{2, 0}+6\mathscr{P}^2_{3, 0}+\mathscr{R}^1_{3, 0}+\mathscr{R}^2_{3, 0}+h.o.t.
\end{eqnarray*}
Here, term \(\mathscr{P}^1_{3, 0}\) is simplified in the third level step.
\end{exm}

\begin{exm} [The infinite level normal form of Example \ref{Exam_CaseIII}] Consider the normalized vector field in Example \ref{Exam_CaseIII}. Here, \(a^k_{0, 1}=b^k_{0, 1}=0\) for \(k=1, 2\) and \({a_{1,0}^1}{a_{1,0}^2}=2>0\). Following Theorem \ref{thm1}, the second level normal form vector field \(v^{(2)}\) in equation \eqref{Simpl319} is the infinite level normal form of \(v^{(1)}\). We remark that the simplest normal form \eqref{Simpl319} can have an infinite number of non-zero higher order terms. Here, we have normalized the system up to degree nine. Thus, we already know that seven terms of degree less than nine from Poincare normal form must be normalized so that the simplest normal form \eqref{Simpl319} is derived from the Poincare normal form \eqref{Example2}.
\end{exm}
\begin{rem} A formal normal form calculations in our approach is always convergent to its infinite normal form in filtration topology; \eg see \cite{GazorYuSpec} and the references therein for the discussion on filtration topology and formal normal form theory. However, the involved power series can be either convergent or divergent. Theoretical investigation of convergence in normal form theory has been a challenging task and it has not been investigated in the literature of hyper-normalization. A numerical investigation of this question requires generating an efficient Maple program so that it numerically calculates the normal forms up to a very high degree; see \cite{GazorMokhtariInt}. Since normal forms are usually truncated at a finite degree for their applications in bifurcation analysis and control, a tedious numerical investigation for the convergence analysis may not seem to be necessary at this stage.
\end{rem}

\end{document}